\documentclass{amsart}

\usepackage{amsmath, amssymb, amsthm}
\usepackage{enumitem}
\usepackage{color}
\usepackage{url}

\newtheorem{theorem}{Theorem}[section]
\newtheorem{lemma}[theorem]{Lemma}
\newtheorem{proposition}[theorem]{Proposition}
\newtheorem{corollary}[theorem]{Corollary}
\newtheorem{claim}[theorem]{Claim}

\theoremstyle{definition}
\newtheorem{definition}[theorem]{Definition}
\newtheorem{remark}[theorem]{Remark}

\newcommand{\cf}{\mathrm{cf}}
\newcommand{\dom}{\mathrm{dom}}
\newcommand{\bb}{\mathbb}

\newcommand{\otp}{\mathrm{otp}}

\newcommand{\tp}{\mathrm{tp}}

\newcommand{\mb}{\mathbf}

\title{Higher-dimensional Delta-systems}
\author{Chris Lambie-Hanson}
\address{Institute of Mathematics of the Czech Academy of Sciences \\ 
\v{Z}itn\'{a} 25, Praha 1, Czechia}
\email{lambiehanson@math.cas.cz}
\urladdr{http://math.cas.cz/lambiehanson}

\begin{document}
\begin{abstract}
  We investigate higher-dimensional $\Delta$-systems indexed by finite sets
  of ordinals, isolating a particular definition thereof and proving a higher-dimensional
  version of the classical $\Delta$-system lemma. We focus in particular on
  systems that consist of sets of ordinals,
  in which case useful order-theoretic uniformities can be
  ensured. We then present three applications of these higher-dimensional
  $\Delta$-systems to problems involving the interplay between forcing and
  partition relations on the reals.
\end{abstract}
\thanks{We thank the anonymous referee for an exceptionally thorough and careful
reading and a great number of corrections and suggestions that have
significantly improved the exposition of the paper.}
\maketitle

\section{Introduction}

The starting point for this paper is one of the basic concepts of combinatorial
set theory: the \emph{$\Delta$-system}.

\begin{definition} \label{classical_def}
  A family $\mathcal{U}$ of sets is a \emph{$\Delta$-system} if there is a
  set $r$, known as the \emph{root} of the $\Delta$-system, such that
  $u \cap v = r$ for all distinct $u,v \in \mathcal{U}$.
\end{definition}

The uniformity provided by $\Delta$-systems can be quite useful, so it is no surprise
that the $\Delta$-system lemma, which isolates conditions that guarantee that a
given family of sets can
be thinned out to form a large $\Delta$-system, is one of the foundational results
of combinatorial set theory. The most commonly stated form of the
lemma, introduced by Shanin \cite{shanin}, is the following.

\begin{lemma}
  Suppose that $\mathcal{U}$ is an uncountable family of finite sets. Then there
  is an uncountable subfamily $\mathcal{U}^* \subseteq \mathcal{U}$ such
  that $\mathcal{U}^*$ is a $\Delta$-system.
\end{lemma}

The following is a less pithy but more general formulation. For a proof,
we direct the reader to \cite[Ch.\ II, \S 1]{kunen}.

\begin{lemma} \label{general_1d_lemma}
  Suppose that $\kappa < \lambda$ are infinite cardinals such that $\lambda$ is
  regular and, for all $\nu < \lambda$, we have $\nu^{<\kappa} < \lambda$.
  Suppose also that $\mathcal{U}$ is a family of sets such that
  $|\mathcal{U}| \geq \lambda$ and $|u| < \kappa$ for all $u \in \mathcal{U}$.
  Then there is $\mathcal{U}^* \subseteq \mathcal{U}$ such that
  $|\mathcal{U}^*| = \lambda$ and $\mathcal{U}^*$ is a $\Delta$-system.
\end{lemma}

$\Delta$-systems are inherently one-dimensional objects, in practice often
enumerated as sequences indexed by ordinals. When investigating higher-dimensional
combinatorial objects, however, one frequently encounters families of sets
indexed by $n$-element sets of ordinals for some $n > 1$ and desires to
find large subfamilies exhibiting certain uniformity properties analogous
to the uniformities exhibited by $\Delta$-systems. In this context,
higher-dimensional analogues of the $\Delta$-system lemma come into play.
Such analogues were first developed in work of Todor\v{c}evi\'{c}
\cite{todorcevic_reals_and_positive_partition_relations} and Shelah
\cite{shelah_sierpinski_ii}, \cite{shelah_positive_partition_theorems}, and
have appeared with increasing frequency of late in works such as
\cite{highly_connected}, \cite{higher_limits}, \cite{dobrinen_hathaway},
\cite{dzamonja_larson_mitchell}, \cite{kumar_raghavan},
\cite{zhang_halpern_lauchli}, and \cite{zhang_monochromatic}.

The higher-dimensional $\Delta$-systems in the aforementioned works
have taken a number of slightly different forms. In this paper, we isolate
one particular definition, based most directly on the 2-dimensional $\Delta$-systems of
\cite{todorcevic_reals_and_positive_partition_relations} and
\cite{highly_connected} and on the $n$-dimensional $\Delta$-systems of
\cite{higher_limits}. This definition generalizes the familiar
1-dimensional definition and, in the case in which the higher-dimensional
$\Delta$-system consists of sets of ordinals, it can be strengthened to incorporate
some additional order-theoretic uniformities.
The definition is presented in Section \ref{def_section}, where we
also prove some basic properties of our higher-dimensional $\Delta$-systems.
In Section \ref{main_result_section}, we prove our main result, Theorem
\ref{general_theorem}, which is an $n$-dimensional analogue of the classical
$\Delta$-system lemma, isolating conditions
under which an $n$-dimensional $\Delta$-system of a particular size can
be guaranteed to exist inside of an arbitrary collection of sets indexed by
$n$-element sets of ordinals. Theorem \ref{general_theorem} is naturally
seen as an elaboration of the Erd\H{o}s-Rado theorem and is closely
connected to the work on canonical
partition relations of Erd\H{o}s and Rado \cite{erdos_rado} and of
Baumgartner \cite{baumgartner}. After proving Theorem \ref{general_theorem},
we turn to a discussion of its optimality, proving that one of its parameters,
the size of the arbitrary collection of sets inside of which we are guaranteed to find
a large $n$-dimensional $\Delta$-system, cannot be improved and indicating
precisely the extent to which another of its parameters, the upper bound on
the size of the members of our arbitrary collection of sets, can consistently
be improved. The results of this section
are summarized in Corollaries \ref{general_cor_1} and \ref{general_cor_2},
which incorporate Theorem \ref{general_theorem}, our discussion of its optimality,
and its connections with the Erd\H{o}s-Rado theorem.

The remaining sections of the paper present applications of our main result.
Section \ref{chain_condition_section} is a
short section presenting a higher-dimensional analogue of the familiar use
of $\Delta$-systems to prove that Cohen forcing satisfies the Knaster
property. In Section \ref{application_section}, we present an
application to a problem
involving the interplay of forcing and polarized partition relations.
In Section \ref{sumset_section}, we show that, in certain arguments, the
$\Delta$-system lemma presented here can successfully replace a different lemma
(from \cite{shelah_sierpinski_ii}) that, at least under the currently
best known results, requires stronger assumptions. We apply this
to a recent result of Zhang \cite{zhang_monochromatic} regarding additive
partition relations on the reals, obtaining a slight local improvement of his
result.

\subsection*{Notation and conventions}

For a class $X$ and a cardinal $\kappa$, $[X]^\kappa := \{Y \subseteq X \mid
|Y| = \kappa\}$, and $[X]^{<\kappa} := \{Y \subseteq X \mid |Y| < \kappa\}$.
For a set $u$ of ordinals, $\otp(u)$ denotes the order type of $u$.
The class of ordinals is denoted by $\mathrm{On}$. If $\rho$ is an ordinal and
$X$ is a class of ordinals, then $[X]^\rho := \{Y \subseteq X \mid
\otp(Y) = \rho\}$. This is a slight abuse of notation given the previous definition
of $[X]^\kappa$ and the customary identification of a cardinal with its initial
ordinal, but in practice we will use the Greek letter $\rho$ precisely when
the order-type definition of $[X]^\rho$ is intended, so no confusion will arise
from this.

We will often think of sets of ordinals as increasing sequences of ordinals
in the natural way. So, for instance, if $u$ is a set of ordinals, $\rho = \otp(u)$,
and $i < \rho$, then $u(i)$ denotes the unique element $\alpha \in u$ such that
$\otp(u \cap \alpha) = i$. If $\mb{i} \subseteq \rho$, then $u[\mb{i}]$ denotes
$\{u(i) \mid i \in \mb{i}\}$. If $X$ is a set of ordinals and $n < \omega$,
then we will use the notation $(\alpha_0, \ldots, \alpha_{n-1}) \in [X]^n$
to denote the conjunction of the statements $\{\alpha_0, \ldots, \alpha_{n-1}\}
\in [X]^n$ and $\alpha_0 < \ldots < \alpha_{n-1}$. If $A$ and $B$ are nonempty sets of
ordinals, then we write $A < B$ to assert that $\alpha < \beta$ for all
$(\alpha, \beta) \in A \times B$. For improved readability, we will also sometimes
omit commas and brackets when using small sets as subscripts or superscripts. For example,
we may write $u^0_{\alpha\beta}$ instead of $u^{\{0\}}_{\{\alpha, \beta\}}$.
For notational convenience, we will adopt the convention that $\max(\emptyset) = -1$.

If $\kappa$ is an infinite cardinal, then $\beth_n(\kappa)$ is defined
by recursion on $n < \omega$ by setting $\beth_0(\kappa) := \kappa$ and
$\beth_{n+1}(\kappa) := 2^{\beth_n(\kappa)}$ for all $n < \omega$.
As is customary, we will denote $\beth_n(\aleph_0)$ simply by $\beth_n$.
Suppose that $\kappa < \lambda$ are cardinals. We say that $\lambda$ is
$\kappa$-inaccessible if $\nu^\kappa < \lambda$ for all $\nu < \lambda$.
Similarly, $\lambda$ is ${<}\kappa$-inaccessible if $\nu^{<\kappa} < \lambda$
for all $\nu < \lambda$.

If $\mu, \lambda$, and $\nu$ are cardinals and $n$ is a natural number, then
the partition relation $\mu \rightarrow (\lambda)^n_\nu$ is the assertion that,
for all $c:[\mu]^n \rightarrow \nu$, there is $H \in [\mu]^\lambda$ such that
$c \restriction [H]^n$ is constant. The negation of this partition relation is
denoted by $\mu \not\rightarrow (\lambda)^n_\nu$.

If $\bb{P}$ is a forcing notion and $p,q \in \bb{P}$, then $p \parallel q$ asserts
that $p$ and $q$ are compatible, i.e., there is $r \in \bb{P}$ such that $r \leq p$
and $r \leq q$, and $p \perp q$ asserts that $p$ and $q$ are incompatible.

\section{Uniform $n$-dimensional $\Delta$-systems} \label{def_section}

In this section, we present the basic definitions of the paper and prove some of
their basic properties. We begin by working towards our definition of an
$n$-dimensional $\Delta$-system indexed by finite sets of ordinals. Most of
the paper will focus on the case in
which the elements of the $\Delta$-system are themselves sets of ordinals, in which
case we can arrange for significant order-theoretic uniformities, but we first
present a more general definition.

Our $n$-dimensional $\Delta$-systems will be indexed by sets of the form $[H]^n$,
where $H$ is a set of ordinals, and for $n > 1$ they will have not a single root
witnessing the fact that they are $n$-dimensional $\Delta$-systems, but rather
a family of roots. When first attempting to generalize $\Delta$-systems
to higher dimensions, one might optimistically hope to require that, in an $n$-dimensional
$\Delta$-system $\langle u_b \mid b \in [H]^n \rangle$, the intersection
$u_a \cap u_b$ depends only on $a \cap b$ for all $a,b \in [H]^n$. In other words,
one might hope to require the existence of a family of roots $\langle R_a \mid
a \in [H]^{\leq n} \rangle$ such that, for all $b, b' \in [H]^n$, we have
$u_b \cap u_{b'} = R_{b \cap b'}$.
However, this would be an overly restrictive
requirement, even in the case of $n=2$. To see this, let $\mu$ be any infinite cardinal,
and define a family of sets $\langle u_b \mid b \in [\mu]^2 \rangle$ by letting
$u_{\alpha\beta} := \{\alpha, \beta + 1\}$ for all $(\alpha, \beta) \in [\mu]^2$.
Now observe that, if $\alpha < \beta < \gamma < \delta < \mu$, then
\begin{itemize}
  \item $\beta \in u_{\beta \gamma} \cap u_{\beta \delta}$;
  \item $\beta \notin u_{\alpha \beta} \cap u_{\beta \gamma}$.
\end{itemize}
Hence, $u_{\beta \gamma} \cap u_{\beta \delta} \neq u_{\alpha \beta}
\cap u_{\beta \gamma}$, yet $\{\beta, \gamma\} \cap \{\beta, \delta\} =
\{\beta\} = \{\alpha, \beta\} \cap \{\beta, \gamma\}$.
Therefore, if one adopts the requirement that $u_a \cap u_b$ must depend
only on $a \cap b$, then one could not even find a subset $H \subseteq \mu$ of
size 4 for which $\langle u_b \mid b \in [H]^n \rangle$ is a 2-dimensional
$\Delta$-system.

The starting point for what will become our actual definition is Todor\v{c}evi\'{c}'s
2-dimensional \emph{double $\Delta$-system} from
\cite{todorcevic_reals_and_positive_partition_relations}. According to Todor\v{c}evi\'{c}'s
definition, if $H$ is a set of ordinals, then a family of sets
$\langle u_b \mid b \in [H]^2 \rangle$ is a \emph{double $\Delta$-system} if
\begin{itemize}
  \item for all $\alpha \in H$, the family $\langle u_{\alpha \beta} \mid
  \beta \in H \setminus (\alpha + 1) \rangle$ is a $\Delta$-system with root
  $r^0_\alpha$ (for simplicity, assume that $H$ has no maximal element);
  \item for all $\beta \in H \setminus \{\min(H)\}$, the family $\langle u_{\alpha\beta} \mid
  \alpha \in H \cap \beta \rangle$ is a $\Delta$-system with root $r^1_\beta$;
  \item both $\langle r^0_\alpha \mid \alpha \in H \rangle$ and
  $\langle r^1_\beta \mid \beta \in H \setminus \{\min(H)\} \rangle$ are $\Delta$-systems,
  with roots $r^0$ and $r^1$, respectively.
\end{itemize}
Note that, if $\langle u_b \mid b \in [H]^2 \rangle$ is a double $\Delta$-system, as
witnessed by sets $\langle r^0_\alpha \mid \alpha \in H \rangle$, $\langle r^1_\beta \mid
\beta \in H \setminus \{\min(H)\} \rangle$, $r^0$, and $r^1$, then it is in fact the case that
$r^0 = r^1 = \bigcap_{b \in [H]^2} u_b$.

In order to succinctly generalize this definition to higher dimensions,
and to help facilitate the later incorporation of further order-theoretic
uniformities, the following notion will be useful.

\begin{definition}
  Suppose that $a$ and $b$ are sets of ordinals.
  \begin{enumerate}
    \item We say that $a$ and $b$ are \emph{aligned} if $\otp(a) = \otp(b)$ and, for all
    $\gamma \in a \cap b$, we have $\otp(a \cap \gamma) = \otp(b \cap \gamma)$.
    In other words, if $\gamma$ is a common element of $a$ and $b$, then it
    occupies the same relative position in both $a$ and $b$.
    \item We let $\mb{r}(a,b) := \{i < \otp(a) \mid a(i) \in b\}$. Notice that
	$a \cap b = a[\mb{r}(a,b)]$ and, if $a$ and $b$ are aligned, then    
    $a \cap b = a[\mb{r}(a,b)] = b[\mb{r}(a,b)]$.
  \end{enumerate}
\end{definition}

Note that, in our counterexample to our initial overly restrictive attempt at
a definition of a higher-dimensional $\Delta$-system at the beginning of this
section, the problem came about when we considered the \emph{non-aligned sets}
$\{\alpha, \beta\}$ and $\{\beta, \gamma\}$. As we will see shortly, it turns out
that this is the only insurmountable problem with our definition, and if one
only requires the family of roots in an $n$-dimensional $\Delta$-system to control
the intersections of elements of the $\Delta$-system indexed by
\emph{aligned} sets, then one obtains a much more workable definition, which
we adopt as our general definition of an $n$-dimensional $\Delta$-system indexed
by $n$-element sets of ordinals.

\begin{definition} \label{non_uniform_def}
  Suppose that $H$ is a set of ordinals, $1 \leq n < \omega$, and, for each
  $b \in [H]^n$, $u_b$ is a set. We call $\langle u_b \mid b \in [H]^n \rangle$
  an \emph{$n$-dimensional $\Delta$-system} if there is a family of \emph{roots}
  $\langle R^{\mb{m}}_a \mid \mb{m} \subseteq n, ~ a \in [H]^{|\mb{m}|} \rangle$
  such that, for all $b, b' \in [H]^n$, if $b$ and $b'$ are aligned and
  $\mb{r}(b, b') = \mb{m}$, then $u_b \cap u_{b'} = R^{\mb{m}}_{b \cap b'}$.
\end{definition}

We observe that, if $n = 1$, then this is precisely the classical definition of a
$\Delta$-system as given in Definition \ref{classical_def} modulo an enumeration of
the $\Delta$-system via a set of ordinals; the root $r$ in Definition
\ref{classical_def} corresponds to the root $R^\emptyset_\emptyset$ in
Definition \ref{non_uniform_def}. When $n=2$, we obtain Todor\v{c}evi\'{c}'s double
$\Delta$-systems; the roots $r^0_\alpha$, $r^1_\alpha$, and $r^0 (=r^1)$ of
Todor\v{c}evi\'{c}'s definition correspond to the roots $R^0_\alpha$, $R^1_\alpha$,
and $R^\emptyset_\emptyset$, respectively.

We now turn to the special setting in which the elements of our $\Delta$-systems
are sets of ordinals. In this setting, we can ask for our $\Delta$-systems to satisfy
certain additional order-theoretic uniformities, and we will call $n$-dimensional
$\Delta$-systems that satisfy these uniformities \emph{uniform $n$-dimensional
$\Delta$-systems}. Since any family of sets
$\langle u_b \mid b \in [H]^n \rangle$ can be transformed into a family of sets of
ordinals via a bijection between $\bigcup_{b \in [H]^n} u_b$ and an ordinal, and
since the proof of our higher dimensional analogue of the $\Delta$-system
lemma (Theorem \ref{general_theorem}) in fact yields 
uniform $n$-dimensional $\Delta$-systems with no additional 
hypotheses, there
will be no loss of generality for us in focusing on this setting.
(However, for $n > 1$ it will \emph{not} in general be the case that every 
sufficiently large 
$n$-dimensional $\Delta$-system consisting of sufficiently small sets of ordinals 
can be refined to a uniform $n$-dimensional $\Delta$-system of the same size; see Remark 
\ref{refinement_remark}.)

Let us first look at the 1-dimensional case to help us motivate our definition.
In the context of families of sets of ordinals, the classical $\Delta$-system
lemma can easily be strengthened to require that the root of the $\Delta$-system
``sits inside" each of its elements in the same way, in the following sense.

\begin{definition} \label{uniform_1d_def}
  A family $\mathcal{U}$ of sets of ordinals is a \emph{uniform $\Delta$-system}
  if there is a set $r$ such that, for all distinct $u,v \in \mathcal{U}$,
  $u$ and $v$ are aligned and $u \cap v = r$.
\end{definition}

The following proposition indicates that Lemma \ref{general_1d_lemma} can be strengthened to 
yield a uniform $\Delta$-system in the case in which $\mathcal{U}$ is a family of sets of ordinals.

\begin{proposition} \label{uniform_1d_prop}
  Suppose that $\kappa < \lambda$ are infinite cardinals such that $\lambda$ is
  regular and ${<}\kappa$-inaccessible. Suppose also that $\mathcal{U}$ is a
  $\Delta$-system consisting of sets of ordinals, and that $|\mathcal{U}| \geq \lambda$ and
  $|u| < \kappa$ for all $u \in \mathcal{U}$. Then there is $\mathcal{U}^*
  \subseteq \mathcal{U}$ such that $|\mathcal{U}^*| = \lambda$ and
  $\mathcal{U}^*$ is a uniform $\Delta$-system.
\end{proposition}

\begin{proof}
  Let $r$ be the root of $\mathcal{U}$.
  Since $\lambda$ is regular and $|u| < \kappa < \lambda$ for all $u \in \mathcal{U}$, by thinning
  out $\mathcal{U}$ if necessary, we may assume that there is an ordinal $\rho < \kappa$
  such that $\otp(u) = \rho$ for all $u \in \mathcal{U}$. Define a function
  $g:\mathcal{U} \rightarrow \mathcal{P}(\rho)$ by letting
  $g(u) := \mb{r}(u, r)$. Since $\lambda$ is
  ${<}\kappa$-inaccessible, we can find a fixed set $\mb{r}^* \subseteq \rho$ and
  a set $\mathcal{U}^* \subseteq \mathcal{U}$ such that $|\mathcal{U}^*| =
  \lambda$ and $g(u) = \mb{r}^*$ for all $u \in \mathcal{U}^*$. Then, for all
  distinct $u,v \in \mathcal{U}^*$, it follows that $u$ and $v$ are aligned,
  with $\mb{r}(u,v) = \mb{r}^*$ and $u \cap v = r$.
\end{proof}

We are now ready for our definition of a uniform $n$-dimensional
$\Delta$-system. In the case $n = 1$, this will coincide with
Definition \ref{uniform_1d_def}, and in the general case it will strengthen Definition
\ref{non_uniform_def} in the same way that Definition \ref{uniform_1d_def}
strengthens Definition \ref{classical_def}.

\begin{definition} \label{general_def}
  Suppose that $H$ is a set of ordinals, $1 \leq n < \omega$, and,
  for all $b \in [H]^n$, $u_b$ is a set of ordinals.
  We call $\langle u_b \mid b \in [H]^n \rangle$ a \emph{uniform
  $n$-dimensional $\Delta$-system} if there is an ordinal $\rho$ and, for
  each $\mb{m} \subseteq n$, a set $\mb{r}_{\mb{m}} \subseteq \rho$
  satisfying the following statements.
  \begin{enumerate}
    \item $\otp(u_b) = \rho$ for all $b \in [H]^n$.
    \item For all $a,b \in [H]^n$ and $\mb{m} \subseteq n$, if $a$ and $b$ are aligned with $\mb{r}(a,b) = \mb{m}$,
    then $u_a$ and $u_b$ are aligned with $\mb{r}(u_a, u_b) = \mb{r}_{\mb{m}}$.
    \item For all $\mb{m}_0, \mb{m}_1 \subseteq n$, we have
    $\mb{r}_{\mb{m}_0 \cap \mb{m}_1} = \mb{r}_{\mb{m}_0} \cap \mb{r}_{\mb{m}_1}$.
  \end{enumerate}
\end{definition}

We now show that Definition \ref{general_def} does indeed strengthen
Definition \ref{non_uniform_def}; clause (1) of the following proposition
will also be useful in a number of other situations.

\begin{proposition} \label{independence_prop}
  Suppose that $1 \leq n < \omega$, $H$ is a set of ordinals, and
  $\langle u_b \mid b \in [H]^n \rangle$ is a uniform $n$-dimensional $\Delta$-system
  as witnessed by an ordinal $\rho$ and sets $\langle \mb{r}_\mb{m} \mid \mb{m}
  \subseteq n \rangle$. Then the following statements hold.
  \begin{enumerate}
    \item For all $\mb{m} \subseteq n$ and all $a,b \in [H]^n$, if
    $a[\mb{m}] = b[\mb{m}]$, then $u_a[\mb{r}_{\mb{m}}] = u_b[\mb{r}_{\mb{m}}]$.
    \item The family $\langle u_b \mid b \in [H]^n \rangle$ is an $n$-dimensional
    $\Delta$-system in the sense of Definition \ref{non_uniform_def}.
  \end{enumerate}
\end{proposition}

\begin{proof}
  (1) For all $a,b \in [H]^n$, let $\partial(a,b) := |\{\beta \in a \cap b \mid
  |a \cap \beta| \neq |b \cap \beta|\}|$. Our proof will be by induction on
  $\partial(a, b)$.

  Fix $\mb{m}$, $a$, and $b$ as in the statement of clause (1) of the proposition. If
  $\partial(a,b) = 0$, then $a$ and $b$ are aligned and $\mb{m} \subseteq
  \mb{r}(a,b)$. It follows from clauses (2) and (3) of Definition \ref{general_def}
  that $u_a$ and $u_b$ are aligned and
  $\mb{r}(u_a, u_b) \supseteq \mb{r}_{\mb{m}}$. In particular,
  $u_a[\mb{r}_{\mb{m}}] = u_b[\mb{r}_{\mb{m}}]$, as desired.

  Now suppose that $\partial(a,b) > 0$ and we have established all instances of
  clause (1) of the proposition for $a', b' \in [H]^n$ for which $a'[\mb{m}] = b'[\mb{m}]$
  and $\partial(a', b') < \partial(a,b)$. Let $\alpha \in a \cap b$ be least
  such that $|a \cap \alpha| \neq |b \cap \alpha|$. Let $k_a, k_b < n$ be such that
  $a(k_a) = \alpha = b(k_b)$. Without loss of generality, we may assume that
  $k_a < k_b$.

  We now alter $a$ to form a new set $a' \in [H]^n$. If $a \cap b \cap \alpha
  \neq \emptyset$, then let $\alpha^* := \max(a \cap b \cap \alpha)$. In this case,
  by our choice of $\alpha$, there must be $k^* < k_a$
  such that $a(k^*) = b(k^*) = \alpha^*$.
  If $a \cap b \cap \alpha = \emptyset$, then let $k^* := -1$. In either case, note
  that, for all $\ell \in (k^*, k_a]$, we have $b(\ell) \notin a$ and,
  if $k^* \geq 0$, then $a(k^*) < b(\ell)$. Moreover, $\mb{m} \cap (k^*, k_a] =
  \emptyset$. We define $a'$ by specifying $a'(\ell)$ for all $\ell < n$.
  If $\ell \leq k^*$ or $\ell > k_a$, then let $a'(\ell) := a(\ell)$. If
  $\ell \in (k^*, k_a]$, then let $a'(\ell) := b(\ell)$. The following observations
  are immediate.
  \begin{enumerate}[label=(\roman*)]
    \item $a$ and $a'$ are aligned, with $\mb{r}(a, a') = n \setminus (k^*, k_a]$.
    In particular, $\mb{m} \subseteq \mb{r}(a, a')$.
    \item $\partial(a', b) = \partial(a, b) - 1$, since
    \[
    \{\beta \in a' \cap b
    \mid |a' \cap \beta| \neq |b \cap \beta|\} = \{\beta \in a \cap b
    \mid |a \cap \beta| \neq |b \cap \beta|\} \setminus \{\alpha\}.
    \]
  \end{enumerate}
  We can therefore invoke the inductive hypothesis together with (i) to conclude that
  $u_a[\mb{r}_\mb{m}] = u_{a'}[\mb{r}_\mb{m}]$ and together with (ii) to conclude
  that $u_{a'}[\mb{r}_\mb{m}] = u_b[\mb{r}_\mb{m}]$, so it follows that
  $u_a[\mb{r}_\mb{m}] = u_b[\mb{r}_\mb{m}]$, as desired.

  (2) To prove that $\langle u_b \mid b \in [H]^n \rangle$ satisfies Definition
  \ref{non_uniform_def}, we must specify roots
  \[
    \langle R^{\mb{m}}_a \mid \mb{m} \subseteq n, ~ a \in [H]^{|\mb{m}|} \rangle.
  \]
  To this end, fix $\mb{m} \subseteq n$ and $a \in [H]^{|\mb{m}|}$. 
  If there are no $b \in [H]^n$ for which $b[\mb{m}] = a$,
  then simply let $R^{\mb{m}}_a := \emptyset$. Otherwise, choose $b \in [H]^n$
  for which $b[\mb{m}] = a$ and set $R^{\mb{m}}_a := u_b[\mb{r}_\mb{m}]$. By
  clause (1) of this proposition, the value of $R^\mb{m}_a$ is independent of
  our choice of $b$.

  Now suppose that $b, b' \in [H]^n$ are aligned and $\mb{r}(b,b') = \mb{m}$, so,
  in particular, $b[\mb{m}] = b \cap b'$. Then
  $u_b$ and $u_{b'}$ are aligned and $\mb{r}(u_b, u_{b'}) = \mb{r}_\mb{m}$.
  Moreover, we defined $R^{\mb{m}}_{b \cap b'}$ so that
  $R^{\mb{m}}_{b \cap b'} = u_b[\mb{r}_\mb{m}]$. It follows that
  $u_b \cap u_{b'} = R^{\mb{m}}_{b \cap b'}$, so $\langle R^{\mb{m}}_a \mid
  \mb{m} \subseteq n, ~ a \in [H]^{|\mb{m}|} \rangle$ witnesses the fact that
  $\langle u_b \mid b \in [H]^n \rangle$ satisfies Definition \ref{non_uniform_def}.
\end{proof}

\begin{remark} \label{refinement_remark}
  There is no direct analogue of Proposition \ref{uniform_1d_prop} for 
  $n$-dimensional $\Delta$-systems when $n > 1$ unless $\lambda$ is weakly 
  compact (see Corollary \ref{wkly_compact_cor} for a positive result in case 
  $\lambda$ is weakly compact). For 
  a simple counterexample, suppose that $\lambda$ is a regular uncountable cardinal 
  that is 
  not weakly compact, let $\pi:\lambda \times \lambda \rightarrow \lambda$ be 
  a bijection, and let $c:[\lambda]^2 \rightarrow 2$ be a function such that 
  $c``[H]^2 = 2$ for all $H \in [\lambda]^\lambda$. Now, for all $\alpha < \beta < 
  \lambda$, let 
  \[
    u_{\alpha\beta} := \begin{cases}
      \emptyset & \text{if } c(\alpha, \beta) = 0 \\
      \pi(\alpha, \beta) & \text{if } c(\alpha, \beta) = 1.
    \end{cases}
  \]
  Then $\langle u_{\alpha\beta} \mid \alpha < \beta < \lambda \rangle$ is 
  a 2-dimensional $\Delta$-system (with $R^0_\alpha = R^1_\alpha = \emptyset$ 
  for all $\alpha < \lambda$) consisting of finite sets of ordinals, yet 
  whenever $H \in [\lambda]^\lambda$, the family $\langle u_{\alpha\beta} \mid 
  (\alpha, \beta) \in [H]^2 \rangle$ contains sets of cardinality 0 and 
  of cardinality 1 and therefore cannot be a uniform 2-dimensional 
  $\Delta$-system.
  
  Nonetheless, as we shall see in Section \ref{main_result_section}, the 
  cardinality hypotheses on the cardinal $\mu$ and the sizes of the 
  sets $u_b$ that guarantee that a family $\langle u_b \mid 
  b \in [\mu]^n \rangle$ of sets of ordinals 
  can be refined to an $n$-dimensional $\Delta$-system 
  of a specified cardinality are already sufficient to guarantee that the 
  family can be refined to a \emph{uniform} $n$-dimensional $\Delta$-system 
  of the same cardinality. 
\end{remark}

\section{A higher-dimensional $\Delta$-system lemma} \label{main_result_section}

In this section, we prove the main result of the paper (Theorem \ref{general_theorem}),
a higher-dimensional analogue of the
$\Delta$-system lemma which asserts, roughly speaking, that inside every
family of sets of ordinals indexed by $n$-element subsets of some sufficiently large
cardinal $\mu$, we can find a subset $H$ of $\mu$ of some specified size such
that $[H]^n$ indexes a uniform $n$-dimensional $\Delta$-system. In the absence
of weakly compact cardinals, this $H$ will necessarily be smaller than $\mu$.
In the same way that the $\Delta$-system lemma can fruitfully be seen as as
an extension of the pigeonhole principle, this $n$-dimensional
$\Delta$-system lemma can fruitfully be seen as an elaboration of the
Erd\H{o}s-Rado theorem, and in fact a version of the Erd\H{o}s-Rado theorem
will be folded into our statement to carry along as an inductive hypothesis.

The result is also closely related to results on \emph{canonical partition relations},
introduced by Erd\H{o}s and Rado in \cite{erdos_rado}, and in particular to
work done by Baumgartner on canonical partition relations \cite{baumgartner},
which can also be seen as an elaboration of the Erd\H{o}s-Rado theorem.
Indeed, in the cases in which $\kappa$ is a successor cardinal, much of our main result
can be derived from the main result of \cite{baumgartner}.
When $\kappa$ is a limit cardinal (and in particular in the important case
$\kappa = \aleph_0$, $\lambda = \aleph_1$), this approach does not
seem to work, so we provide a single proof that covers all cases.
We first introduce the following notation, from \cite{baumgartner}, that allows
us to indicate precisely the size of the family needed to ensure the existence of a large
uniform $n$-dimensional $\Delta$-system.

\begin{definition} \label{sigma_def}
  Given an infinite regular cardinal $\lambda$, recursively define $\sigma(\lambda, n)$ for
  $1 \leq n < \omega$ by letting $\sigma(\lambda, 1) := \lambda$ and, given
  $1 \leq n < \omega$, letting $\sigma(\lambda, n + 1) := \left(2^{<\sigma(\lambda, n)}
  \right)^+$.
\end{definition}

\begin{remark} \label{sigma_remark}
  To connect Definition \ref{sigma_def} with the already familiar $\beth$-notation
  and to help clarify the choice of cardinals in the statements of Corollary
  \ref{aleph_1_cor}, Theorem \ref{polarized_forcing_theorem}, and Corollary
  \ref{additive_cor}, we make the following observations, which we leave the reader
  to verify.
  \begin{enumerate}
    \item If $\lambda = \kappa^+$ and $1 \leq n < \omega$, then $\sigma(\lambda,
    n) = (\beth_{n-1}(\kappa))^+$. In particular, $\sigma(\aleph_1, n) =
    \beth_{n-1}^+$ and $\sigma(\beth_1^+, n) = \beth_n^+$.
    \item For every infinite regular $\lambda$, if $2 \leq n < \omega$, then
    $\sigma(\lambda, n) = (\beth_{n-2}(2^{<\lambda}))^+$.
  \end{enumerate}
  Note in particular that $\sigma(\lambda, n)$ is regular for each regular infinite
  $\lambda$ and each $1 \leq n < \omega$.

  We also remark that $\sigma(\lambda, n)$ is precisely the cardinal resource
  needed to ensure a monochromatic set of size $\lambda$ in the $n$-dimensional
  Erd\H{o}s-Rado theorem, which can be formulated as follows: for every
  $1 \leq n < \omega$ and all infinite cardinals $\nu < \lambda$, with
  $\lambda$ regular, the partition relation $\sigma(\lambda, n) \rightarrow
  (\lambda + (n-1))^n_\nu$ holds (\cite[Theorem 39]{erdos_rado_partition_calculus};
  cf.\ also \cite[Proposition 1]{baumgartner}). See Corollary \ref{general_cor_1}
  for a more precise formulation of the connection between our main result and the
  Erd\H{o}s-Rado theorem.
\end{remark}

In the proof of Theorem \ref{general_theorem}, we will make use of the following notion of the
\emph{type} of a sequence of sets of ordinals, which describes the order-relations
existing among the sets.

\begin{definition} \label{type_def}
  Suppose that $I$ is a set and, for all $i \in I$, $u_i$ is a set of ordinals.
  Then $\tp(\langle u_i \mid i \in I \rangle)$ (the \emph{type of
  $\langle u_i \mid i \in I \rangle$}) is a function from
  $\otp(\bigcup_{i \in I}u_i)$ to $\mathcal{P}(I)$ defined as follows.
  First, let $\bigcup_{i \in I}u_i$ be enumerated in increasing order as
  $\langle \alpha_\eta \mid \eta < \otp(\bigcup_{i \in I}u_i) \rangle$.
  Then, for all $\eta < \otp(\bigcup_{i \in I}u_i)$, let
  $\tp(\langle u_i \mid i \in I \rangle)(\eta) := \{i \in I \mid
  \alpha_\eta \in u_i\}$.
\end{definition}

We will often slightly abuse notation and write, for instance,
$\tp(u_0, u_1, u_2)$ instead of $\tp(\langle u_0, u_1, u_2 \rangle)$.

\begin{remark} \label{type_remark}
  To connect Definition \ref{type_def} with the earlier definition of
  \emph{aligned sets}, we note that, if $a$ and $b$ are sets of ordinals, then
  $a$ and $b$ are aligned if and only if $\tp(a \cap b, a) = \tp(a \cap b, b)$.
  We also observe the following useful facts about the $\tp$ operator, which
  can easily be verified:
  \begin{enumerate}
    \item Suppose that $I$ is a set and, for all $i \in I$, $u_i$ and $u'_i$ are
    sets of ordinals. Suppose also that $\tp(\langle u_i \mid i \in I \rangle) = \tp(\langle
    u_i' \mid i \in I \rangle)$. Then the following statements hold.
    \begin{enumerate}
      \item For all $i \in I$, we have $\otp(u_i) = \otp(u'_i)$.
      \item For all $J \subseteq I$, we have $\tp(\langle u_i \mid i \in J \rangle)
      = \tp(\langle u'_i \mid i \in J \rangle)$.
    \end{enumerate}
    \item Suppose that $u_0$, $u_1$, $u_0'$, and $u_1'$ are sets of ordinals.
    If $u_0$ and $u_1$ are aligned and $\tp(u_0, u_1) = \tp(u_0', u_1')$,
    then $u_0'$ and $u_1'$ are also aligned and $\mb{r}(u_0', u_1') =
    \mb{r}(u_0, u_1)$.
  \end{enumerate}
\end{remark}

The higher-dimensional $\Delta$-systems that we isolate in our main result will have
an additional technical uniformity (the ``moreover" clause of
Theorem \ref{general_theorem}) that allows us to control the relationship between
$u_a$ and $u_b$ for certain non-aligned pairs $a,b \in [H]^n$ and is useful
in some applications. In order to properly
state it, we need some further definitions. Readers can safely skip these technical
considerations and the ``moreover" clause of the theorem on first read, if desired, as
they are not needed in our applications in Sections
\ref{chain_condition_section} and \ref{application_section}. They are
used in the proof of Corollary \ref{additive_cor}, which is presented not in this
paper but in \cite{sumset_proof_alteration}.

\begin{definition}
  Suppose that $i < \rho$ are ordinals and $a, b \in [\mathrm{On}]^\rho$. We say
  that $a$ and $b$ are \emph{aligned above $i$} if $a[\rho \setminus i]$ and
  $b[\rho \setminus i]$ are aligned.
\end{definition}

The following notion provides strictly less information than $\tp(a,b)$ but
is sometimes easier to control.

\begin{definition}
  Suppose that $a$ and $b$ are sets of ordinals. Then the \emph{intersection type
  of $a$ and $b$}, denoted $\tp_{\mathrm{int}}(a,b)$, is the set
  $\{(i,j) \in \otp(a) \times \otp(b) \mid a(i) = b(j)\}$.
\end{definition}

\begin{definition}
  Suppose that $a$ is a nonempty set of ordinals and $i < \otp(a)$.
  \begin{enumerate}
    \item We say that an ordinal $\alpha$ is \emph{$i$-possible} for $a$ if
    the following two statements hold:
    \begin{enumerate}
      \item if $i > 0$, then $\alpha > a(i-1)$;
      \item if $i + 1 < \otp(a)$, then $\alpha < a(i+1)$.
    \end{enumerate}
    Intuitively, $\alpha$ is $i$-possible for $a$ if $a(i)$ can be replaced
    by $\alpha$ without changing the relative positions of the other elements
    of $a$.
    \item If $\alpha$ is $i$-possible for $a$, then $a_{i \mapsto \alpha}$ is the
    set $\left(a \setminus \{a(i)\}\right) \cup \{\alpha\}$, i.e., the set obtained by replacing
    the $i^{\mathrm{th}}$ element of $a$ with $\alpha$.
  \end{enumerate}
\end{definition}

We are now ready for our main result. As we will see
at the end of this section, unless $\lambda$ is a weakly compact cardinal,
the theorem is optimal in the sense that $\mu$ cannot be lowered. We also note
that clause (1) of the following theorem is essentially the Erd\H{o}s-Rado theorem.
In response to a query from the
referee, we note that our proof does not yield an essentially new proof of the
Erd\H{o}s-Rado theorem; if one extracts the proof of just clause (1) from our
proof, one obtains more or less a proof of the Erd\H{o}s-Rado theorem originally
given by Simpson in \cite{simpson} (see also the proof of
\cite[Theorem 7.2.1]{chang_keisler}).

\begin{theorem} \label{general_theorem}
  Suppose that
  \begin{itemize}
    \item $1 \leq n < \omega$;
    \item $\kappa, \nu < \lambda$ are infinite cardinals, $\lambda$ is regular and
    ${<}\kappa$-inaccessible, and $\mu = \sigma(\lambda, n)$;
    \item $g: [\mu]^n \rightarrow \nu$;
    \item for all $b \in [\mu]^n$, we are given a set $u_b \in [\mathrm{On}]^{< \kappa}$.
  \end{itemize}
  Then there are $H \in [\mu]^\lambda$ and $k < \nu$ such that
  \begin{enumerate}
    \item $g(b) = k$ for all $b \in [H]^n$;
    \item $\langle u_b \mid b \in [H]^n \rangle$ is a uniform $n$-dimensional
    $\Delta$-system.
  \end{enumerate}
  Moreover, we can arrange our choice of $H$ so that,
  for all $a,b \in [H]^n$ and all $m < n$, if it is the
  case that $a$ and $b$ are aligned above $m$ and $a(m) = b(m)$, then, for any
  ordinal $\alpha \in H$ that is $m$-possible for both $a$ and $b$, we have $\tp_{\mathrm{int}}(u_a,u_b)
  = \tp_{\mathrm{int}}(u_{a_{m \mapsto \alpha}}, u_{b_{m \mapsto \alpha}})$.
\end{theorem}

\begin{proof}
  The proof is by induction on $n$. When $n = 1$, the result follows from
  Proposition \ref{uniform_1d_prop} and the pigeonhole principle (note that
  the ``moreover" clause of the theorem is trivial if $n=1$). So suppose that
  $1 < n < \omega$ and we have established all instances of the theorem for $n - 1$.

  Set $\mu^* := \sigma(\lambda, n-1)$. We will construct the desired set $H$ via a sequence of 
  refinements, which we outline here at the start. We will first isolate an ordinal $\mu_M < \mu$, of 
  size $2^{<\mu^*}$ and cofinality at least $\mu^*$. Next, we will build a set $A \subseteq \mu_M$ 
  of order type $\mu^*$ exhibiting certain uniformities with respect to the family 
  $\langle u_b \mid b \in [\mu]^n \rangle$ and the function $g$.
  An application of the inductive 
  hypothesis for $n-1$ will then yield a set $H_0 \subseteq A$ of cardinality $\lambda$. 
  Finally, we will thin out $H_0$ one last time by recursively constructing an increasing sequence 
  $\langle \beta_\xi \mid \xi < \lambda \rangle$ from $H_0$ and letting $H := \{\beta_\xi \mid \xi 
  < \lambda\}$. Together, this sequence of refinements is as follows:
  \[
  \mu \supseteq \mu_M \supseteq A \supseteq H_0 \supseteq 
  \{\beta_\xi \mid \xi < \lambda\} = H.
  \]
  
  To begin, let $\theta$ be a sufficiently large regular cardinal, and let $M$ be an
  elementary substructure of $(H(\theta), \in,
  g, \langle u_b \mid b \in [\mu]^n \rangle)$ such that
  $M$ is closed under sequences of length less than $\mu^*$
  and $\mu_M := M \cap \mu \in \mu$.
  This is possible, since $\mu^*$ is regular and $\mu =
  \sigma(\lambda, n) = \left(2^{<\mu^*} \right)^+$. Note that
  $\cf(\mu_M) \geq \mu^*$.

  Temporarily fix an arbitrary $a \in [\mu_M]^{n-1}$, and consider $u_{a^\frown \langle \mu_M \rangle}$.
  Let $w_a := u_{a ^\frown \langle \mu_M \rangle} \cap M$ and
  $\rho_a = \otp\left(u_{a^\frown \langle \mu_M \rangle}\right)$. Let $\mb{i}_a := \mb{r}(u_{a^\frown 
  \langle \mu_M \rangle}, w_a)$, and let $\mb{j}_a := \rho_a \setminus \mb{i}_a$. Note 
  that $u_{a^\frown \langle \mu_M \rangle}[\mb{i}_a] = w_a$. For each
  $j \in \mb{j}_a$, let $\gamma_{a,j}$ be the least ordinal $\gamma$ in $M$
  such that $u_{a^\frown \langle \mu_M \rangle}(j) < \gamma$; to see that such
  an ordinal $\gamma$ exists, note that $\sup(\bigcup_{b \in [\mu]^n} u_b)$ is
  definable in $M$ and is therefore an element of $M$.
  
  \begin{claim} \label{a_claim}
    There is a set $A \subseteq \mu_M$ of order type $\mu^*$ such that:
    \begin{enumerate}
      \item For every $a \in [A]^{n-1}$ and every $\beta \in A$ with $\max(a) < \beta$:
      \begin{enumerate}
        \item $g(a^\frown \langle \beta \rangle) = g(a^\frown \langle \mu_M \rangle)$;
        \item $\otp(u_{a^\frown \langle \beta \rangle}) = \rho_a$;
        \item $u_{a^\frown \langle \beta \rangle}[\mb{i}_a] = w_a$.
      \end{enumerate}
      \item For every $a \in [A]^{n-1}$, all $\alpha, \beta \in A$ with $\max(a) < \alpha < \beta$, 
      and all $j \in \mb{j}_a$, we have $u_{a^\frown \langle \beta \rangle}(j) \notin 
      u_{a^\frown \langle \alpha \rangle}$.
      \item For every $\beta \in A$, we have \[
      \tp(\langle u_{a^\frown \langle \beta \rangle} \mid 
      a \in [A \cap \beta]^{n-1} \rangle) = \tp(\langle u_{a^\frown \langle \mu_M} \rangle \mid 
      a \in [A \cap \beta]^{n-1} \rangle).
      \]
      In particular, if $a_0, a_1 \in [A \cap \beta]^{n-1}$, then 
      \[
      \tp(u_{a_0{}^\frown \langle \beta 
      \rangle}, u_{a_1{}^\frown \langle \beta \rangle}) = \tp(u_{a_0{}^\frown \langle \mu_M \rangle}, 
      u_{a_1{}^\frown \langle \mu_M \rangle}).
      \]
    \end{enumerate}
  \end{claim}
  
  \begin{proof}
  We will recursively construct an increasing sequence $\langle \alpha_\eta \mid \eta <
  \mu^* \rangle$ of ordinals below $\mu_M$ and then let $A := \{\alpha_\eta \mid \eta < \mu^* \}$.
  Our construction will maintain the hypothesis that, for all $\eta < \mu^*$, $A_\eta := 
  \{\alpha_\xi \mid \xi < \eta\}$ satisfies all of the items in the statement of the claim.
  
  Begin by letting $\alpha_\eta = \eta$ for all $\eta < n - 1$. Now suppose that
  $n - 1 \leq \eta < \mu^*$ and we have defined $\langle
  \alpha_\xi \mid \xi < \eta \rangle$.
  By the closure of $M$ and the fact that $[A_\eta]^{n-1}$ has size less than
  $\mu^*$, we know that all of the following are elements of $M$:
  \begin{itemize}
    \item $A_\eta$;
    \item $\langle g(a ^\frown \langle \mu_M \rangle) \mid a \in [A_\eta]^{n-1}
    \rangle$;
    \item $\langle (w_a, \rho_a, \mb{i}_a, \mb{j}_a) \mid a \in [A_\eta]^{n-1} \rangle$;
    \item $\langle \gamma_{a,j} \mid a \in [A_\eta]^{n-1}, ~ j \in \mb{j}_\alpha \rangle$.
  \end{itemize}
  Moreover, $\tp(\langle u_{a ^\frown \langle \mu_M \rangle} \mid a \in
  [A_\eta]^{n-1} \rangle)$ is a function from an ordinal less than
  $\mu^*$ to $\mathcal{P}([A_\eta]^{n-1})$, so again the closure of
  $M$ implies that $\tp(\langle u_{a ^\frown \langle \mu_M \rangle} \mid a \in
  [A_\eta]^{n-1} \rangle)$ is in $M$.

  For each $a \in [A_\eta]^{n-1}$ and $j \in \mb{j}_a$, let
  \[
    \epsilon_{a,j} := \sup\{\sup(u_b \cap \gamma_{a,j}) \mid b \in [A_\eta]^n \}.
  \]
  Note that $\cf(\gamma_{a,j}) \geq \mu^*$, since otherwise
  there would be a cofinal $x \subseteq \gamma_{a,j}$ such that $x \subseteq M$.
  Therefore, we have $\epsilon_{a,j} \in M \cap \gamma_{a,j}$ and, again by closure,
  $\langle \epsilon_{a,j} \mid a \in [A_\eta]^{n-1}, ~ j \in \mb{j}_a \rangle \in M$.

  In $H(\theta)$, the ordinal $\mu_M$ witnesses the truth of the statement
  asserting the existence of an ordinal $\beta$ such that:
  \begin{itemize}
    \item $\sup(A_\eta) < \beta < \mu$;
    \item $g(a ^\frown \langle \beta \rangle) = g(a^\frown \langle \mu_M \rangle)$
    for all $a \in [A_\eta]^{n-1}$;
    \item $\tp(\langle u_{a ^\frown \langle \beta \rangle} \mid a \in
    [A_\eta]^{n-1} \rangle ) = \tp(\langle u_{a ^\frown \langle \mu_M \rangle} \mid a \in
    [A_\eta]^{n-1} \rangle )$;
    \item $u_{a^\frown \langle \beta \rangle}[\mb{i}_a] = w_a$ for all $a \in
    [A_\eta]^{n-1}$;
    \item $u_{a ^\frown \langle \beta \rangle}(j)$ is in the interval
    $(\epsilon_{a,j}, \gamma_{a,j})$ for all $a \in [A_\eta]^{n-1}$ and
    all $j \in \mb{j}_a$.
  \end{itemize}
  All of the parameters in the above statement are in $M$ (note, for instance,
  that, in the second item, $a^\frown \langle \mu_M \rangle$ is not in $M$,
  but $\langle g(a^\frown \langle \mu_M \rangle) \mid a \in [A_\eta]^{n-1}\rangle$ is).
  Therefore, by elementarity, we can choose $\alpha_\eta \in M$ satisfying the statement.
  It is evident that this choice of $\alpha_\eta$ satisfies the requirements of the construction. 
  In particular, notice that, as a
  consequence of clause (1a) of Remark \ref{type_remark} and the fact that
  $\alpha_\eta$ satisfies the third bullet point above, we have
  $\otp(u_{a ^\frown \langle \alpha_\eta \rangle}) = \otp(u_{a ^\frown \langle \mu_M \rangle})
  = \rho_a$ for all $a \in [A_\eta]^{n-1}$. Also, the last bullet point above ensures 
  that, for all $a \in [A_\eta]^{n-1}$, all $\alpha \in A_\eta \setminus (\max(a) + 1)$, and all 
  $j \in \mb{j}_a$, we have $u_{a^\frown \langle \alpha_\eta \rangle}(j) \notin 
  u_{a^\frown \langle \alpha \rangle}$.
  Therefore, this completes the construction and the proof of the claim.
  \end{proof}

  Let $A$ be as given by Claim \ref{a_claim}. Define a function $g^*$ on $[A]^{n-1}$ by
  letting $g^*(a) := \left\langle g(a^\frown \langle \mu_M \rangle), \rho_a,
  \mb{i}_a, \mb{j}_a \right\rangle$ for all $a \in [A]^{n-1}$.
  Since we know that
  \begin{itemize}
    \item $g: [\mu]^n \rightarrow \nu$;
    \item $\rho_a < \kappa$; and
    \item $\mb{i}_a, \mb{j}_a \subseteq \rho_a$;
  \end{itemize}
  it follows that $g^*$ can be coded as a function from $[A]^{n-1}$ to
  $\max\{\nu, 2^{<\kappa}\}$ which, by the hypothesis of the theorem, is
  less than $\lambda$. Recalling that $\mu^* = \sigma(\lambda, n-1) = |A|$, apply the
  induction hypothesis to $g^*$ and
  $\langle u_{a^\frown \langle \mu_M \rangle} \mid a \in [A]^{n-1} \rangle$
  to find $H_0 \subseteq A$, $k < \nu$, $\rho < \kappa$, and sets
  $\mb{i}, \mb{j} \subseteq \rho$ such that the following statements all hold:
  \begin{itemize}
    \item $\otp(H_0) = \lambda$;
    \item $g(a ^\frown \langle \mu_M \rangle) = k$ for all $a \in [H_0]^{n-1}$;
    \item $\langle \rho_a, \mb{i}_a, \mb{j}_a \rangle = \langle \rho,
    \mb{i}, \mb{j} \rangle$ for all $a \in [H_0]^{n-1}$;
    \item $\langle u_{a ^\frown \langle \mu_M \rangle} \mid a \in [H_0]^{n-1} \rangle$
    is a uniform $(n-1)$-dimensional $\Delta$-system, as witnessed by $\rho$ and by
    sets $\mb{s}_{\mb{m}} \subseteq \rho$ for each $\mb{m} \subseteq n - 1$;
    \item $\langle u_{a ^\frown \langle \mu_M \rangle} \mid
    a \in [H_0]^{n-1} \rangle$ satisfies the ``moreover" clause in the statement
    of the theorem.
  \end{itemize}

  We will thin out $H_0$ to a further unbounded
  subset $H \subseteq H_0$ before the end of the proof. For now, let us begin
  verifying clauses (1) and (2) in the statement of the theorem, noting that
  what we verify for $H_0$ will remain true after further thinning out.

  We first take care of clause (1) of the theorem, simultaneously showing that
  $\otp(b) = \rho$ for all $b \in [H_0]^n$. To this end, fix $b \in [H_0]^n$.
  Then $b$ is of the form
  $a ^\frown \langle \beta \rangle$ for some $\beta \in 
  H_0$ and $a \in [H_0 \cap \beta]^{n-1}$.
  Since $A$ satisfies Clause (1) of Claim \ref{a_claim} and $H_0 \subseteq A$, we have
  $g(a ^\frown \langle \beta \rangle) = g(a ^\frown \langle \mu_M \rangle)$,
  and $\otp(u_{a ^\frown \langle \beta \rangle}) =
  \otp(u_{a ^\frown \langle \mu_M \rangle}) = \rho_a$.
  Then, by our choice of $H_0$, $k$, and $\rho$, we have $g(a ^\frown \langle \mu_M \rangle)
  = k$ and $\rho_a = \rho$. Therefore, $g(b) = k$
  and $\otp(u_b) = \rho$, as desired.

  We now turn our attention to clause (2). The value of $\rho$ that we isolated
  above is the order type that will eventually witness that $\langle u_b \mid b \in [H]^n
  \rangle$ is a uniform $n$-dimensional $\Delta$-system; indeed, by the previous
  paragraph we have $\otp(u_b) = \rho$ for all $b \in [H_0]^n$.
  We next specify the values
  for $\langle \mb{r}_\mb{m} \mid \mb{m} \subseteq n \rangle$ that will
  witness that $\langle u_b \mid b \in [H]^n \rangle$ is a uniform
  $n$-dimensional $\Delta$-system. For each $\mb{m} \subseteq n$, let
  $\mb{m}^- := \mb{m} \cap (n-1)$. If $n-1 \in \mb{m}$, then set
  $\mb{r}_{\mb{m}} := \mb{s}_{\mb{m}^-}$. If $n-1 \notin \mb{m}$, then set
  $\mb{r}_{\mb{m}} := \mb{s}_{\mb{m}^-} \cap \mb{i}$. Note that, in either case,
  we do indeed have $\mb{r}_{\mb{m}} \subseteq \rho$.

  \begin{claim} \label{claim_214}
    For all $\mb{m}_0, \mb{m}_1 \subseteq n$, we have $\mb{r}_{\mb{m}_0
    \cap \mb{m}_1} = \mb{r}_{\mb{m}_0} \cap \mb{r}_{\mb{m}_1}$.
  \end{claim}

  \begin{proof}
    This follows immediately from the fact that $\mb{s}_{\mb{m}_0^-
    \cap \mb{m}_1^-} = \mb{s}_{\mb{m}_0^-} \cap \mb{s}_{\mb{m}_1^-}$
    for all $\mb{m}_0, \mb{m}_1 \subseteq n$.
  \end{proof}

  It remains to verify clause (2) of Definition \ref{general_def}, i.e.,
  if $a,b \in [H]^n$ are aligned and $\mb{r}(a,b) = \mb{m}$, then
  $u_a$ and $u_b$ are aligned, and $\mb{r}(u_a, u_b) = \mb{r}_\mb{m}$. We split
  this verification into two cases, depending on whether or not $n-1$ is
  in $\mb{m}$.

  \begin{claim} \label{top_element_claim}
    Suppose that $b_0, b_1 \in [H_0]^n$ are aligned and $n-1 \in \mb{m} =
    \mb{r}(b_0,b_1)$. Then $u_{b_0}$ and $u_{b_1}$ are aligned and
    $\mb{r}(u_{b_0}, u_{b_1}) = \mb{r}_\mb{m}$.
  \end{claim}

  \begin{proof}
    Since $n-1 \in \mb{m}$, we have $\mb{r}_{\mb{m}} = \mb{s}_{\mb{m}^-}$.
    It also follows from the fact
    that $n-1 \in \mb{m}$ that there is $\beta \in H_0$ such that
    $b_0$ and $b_1$ are of the form $a_0{}^\frown \langle \beta \rangle$
    and $a_1{}^\frown \langle \beta \rangle$ respectively, where
    $a_0, a_1 \in [H_0 \cap \beta]^{n-1}$ are aligned and $\mb{r}(a_0, a_1) = \mb{m}^-$.
    By our choice of $H_0$ and $\mb{s}_{\mb{m}^-}$, it follows that
    $u_{a_0{}^\frown \langle \mu_M \rangle}$ and $u_{a_1{}^\frown \langle
    \mu_M \rangle}$ are aligned and $\mb{r}(u_{a_0{}^\frown \langle \mu_M \rangle},
    u_{a_1{}^\frown \langle \mu_M \rangle}) = \mb{s}_{\mb{m}^-}$.
    The fact that $A$ satisfies Clause (3) of Claim \ref{a_claim} then implies
    that $\tp(u_{b_0}, u_{b_1}) =
    \tp(u_{a_0{}^\frown \langle \mu_M \rangle}, u_{a_1{}^\frown \langle
    \mu_M \rangle})$, and therefore, recalling Remark \ref{type_remark},
    that $u_{b_0}$ and $u_{b_1}$ are aligned,
    with $\mb{r}(u_{b_0}, u_{b_1}) = \mb{s}_{\mb{m}^-} = \mb{r}_\mb{m}$, as
    desired.
  \end{proof}

  We next deal with the case in which $\mb{m} \subseteq n-1$. This will
  take a bit more work. We first establish the following claim.

  \begin{claim} \label{no_top_element_claim}
    Suppose that $b_0, b_1 \in [H_0]^n$, $\mb{m} \subseteq n-1$, and
    $b_0[\mb{m}] = b_1[\mb{m}]$. Then $u_{b_0}[\mb{r}_\mb{m}] =
    u_{b_1}[\mb{r}_\mb{m}]$.
  \end{claim}

  \begin{proof}
    Since $\mb{m} \subseteq n-1$, we have $\mb{m}^- = \mb{m}$ and $\mb{r}_\mb{m}
    = \mb{s}_\mb{m} \cap \mb{i}$. We also know that $b_0$ and $b_1$ are of the form
    $a_0{}^\frown \langle \alpha \rangle$ and $a_1{}^\frown
    \langle \beta \rangle$, respectively, where
    $\alpha, \beta \in H_0$, $a_0, a_1 \in [H_0]^{n-1}$, and
    $a_0[\mb{m}] = a_1[\mb{m}]$. By Proposition \ref{independence_prop}(1)
    applied to $\langle u_{a ^\frown \langle \mu_M \rangle} \mid a \in [H_0]^{n-1} \rangle$,
    $\mb{m}$, $a_0$, and $a_1$, we know that $u_{a_0{}^\frown \langle \mu_M
    \rangle}[\mb{s}_\mb{m}] = u_{a_1{}^\frown \langle \mu_M \rangle}[\mb{s}_\mb{m}]$.
    Now fix $i \in \mb{r}_\mb{m}$. Since $i \in \mb{s}_\mb{m}$,
    it follows that $u_{a_0{}^\frown \langle \mu_M \rangle}(i) = u_{a_1{}^\frown \langle \mu_M
    \rangle}(i)$. Since $i \in \mb{i}$, the fact that $A$ satisfies Clause (1c) of Claim \ref{a_claim}
    implies that $u_{b_0}(i) = u_{a_0{}^\frown \langle \mu_M \rangle}(i)$ and $u_{b_1}(i) =
    u_{a_1{}^\frown \langle \mu_M \rangle}(i)$. Together, this implies that
    $u_{b_0}(i) = u_{b_1}(i)$, and hence $u_{b_0}[\mb{r}_\mb{m}] =
    u_{b_1}[\mb{r}_\mb{m}]$.
  \end{proof}

  As an immediate consequence of Claim \ref{no_top_element_claim}, if $b_0, b_1 \in
  [H_0]^n$ are aligned and $\mb{r}(b_0, b_1) = \mb{m} \subseteq n-1$, then
  $u_{b_0}[\mb{r}_\mb{m}] = u_{b_1}[\mb{r}_\mb{m}]$.
  Showing that $u_{b_0}$ and $u_{b_1}$ are disjoint outside of
  $u_{b_0}[\mb{r}_{\mb{m}}]$ will take some more work and possibly a thinning out of $H_0$.
  For $m < n$ and $a \in [H_0]^m$, choose any $b \in [H_0]^n$ with
  $a = b[m]$ (i.e., $b$ is an end-extension of $a$), and define $u_a := u_b[\mb{r}_m]$.
  By Claim \ref{no_top_element_claim}, this definition is independent of our choice
  of $b$.

  For the following claim, recall our convention that $\max(\emptyset) = -1$.

  \begin{claim} \label{1d_claim}
    Suppose that $m < n$ and $a \in [H_0]^m$. Then
    \[
      \langle u_{a ^\frown \langle \beta \rangle} \mid \beta \in H_0 \setminus
      (\max(a) + 1) \rangle
    \]
    is a $\Delta$-system with root $u_a$.
  \end{claim}

  \begin{proof}
    Suppose first that $m = n-1$, in which case $\mb{r}_m = \mb{i}$.
    Fix $(\alpha, \beta) \in [H_0]^2$ with $\alpha > \max(a)$,
    and consider $u_{a ^\frown \langle \alpha
    \rangle} \cap u_{a ^\frown \langle \beta \rangle}$.
    By Claim \ref{no_top_element_claim}, we have
    $u_{a ^\frown \langle \alpha \rangle}[\mb{i}] = u_{a ^\frown \langle
    \beta \rangle}[\mb{i}]$. Furthermore, for all $j \in \mb{j}$,
    the fact that $A$ satisfies Clause (2) of Claim \ref{a_claim}
    implies that $u_{a ^\frown \langle \beta \rangle}(j) \notin
    u_{a ^\frown \langle \alpha \rangle}$. It follows that
    \[
      u_{a ^\frown \langle \alpha \rangle} \cap u_{a ^\frown \langle
      \beta \rangle} = u_{a ^\frown \langle
      \alpha \rangle}[\mb{i}] = u_a,
    \]
    as desired.

    Next, suppose that $m < n - 1$. Fix $(\beta_0, \beta_1) \in
    [H_0]^2$ with $\beta_0 > \max(a)$, and consider $u_{a ^\frown
    \langle \beta_0 \rangle} \cap u_{a ^\frown \langle \beta_1 \rangle}$.
    Fix $c \in [H_0]^{n-m-1}$ with $\min(c) > \beta_1$ and set
    $b_\ell := a ^\frown \langle \beta_\ell \rangle ^\frown c$ for $\ell < 2$.
    Note that $b_\ell \in [H_0]^n$, that
    $u_{a^\frown \langle \beta_\ell \rangle} = u_{b_\ell}[\mb{r}_{m+1}]$, and that
    $u_a = u_{b_\ell}[\mb{r}_m]$. Observe also that $b_0$ and $b_1$ are aligned
    and that $\mb{r}(b_0, b_1) = n \setminus \{m\}$, so, by Claim \ref{top_element_claim},
    we have $u_{b_0} \cap u_{b_1} = u_{b_0}[\mb{r}_{n \setminus \{m\}}] =
    u_{b_1}[\mb{r}_{n \setminus \{m\}}]$. Putting this together,
    we obtain
    \begin{align*}
      u_{a ^\frown \langle \beta_0 \rangle} \cap u_{a ^\frown \langle \beta_1
      \rangle} &= u_{b_0}[\mb{r}_{m+1}] \cap u_{b_1}[\mb{r}_{m+1}] \\
      &= u_{b_0}[\mb{r}_{m+1}] \cap u_{b_1}[\mb{r}_{m+1}] \cap u_{b_0}[\mb{r}_{n
      \setminus \{m\}}] \cap u_{b_1}[\mb{r}_{n \setminus \{m\}}] \\
      &= u_{b_0}[\mb{r}_m] \cap u_{b_1}[\mb{r}_m] \\
      &= u_a,
    \end{align*}
    where the passage from the second to the third line in the above sequence
    of equations follows from Claim \ref{claim_214} and the observation that
    $(m+1) \cap (n \setminus \{m\}) = m$.
  \end{proof}

  We are now ready to thin out $H_0$ to our final set $H$ witnessing the
  conclusion of the theorem. We will recursively construct an increasing
  sequence $\langle \beta_\xi \mid \xi < \lambda \rangle$
  of ordinals from $H_0$ and then define $H := \{\beta_\xi \mid
  \xi < \lambda \}$.

  Begin by letting $\beta_0 := \min(H_0)$. Next, suppose that $0 < \zeta <
  \lambda$ and $\langle \beta_\xi \mid \xi < \zeta \rangle$
  has been defined. Let $B_\zeta := \{\beta_\xi \mid \xi < \zeta \}$.
  Suppose that $a_0 \in [B_\zeta]^{<n}$ and $a_1 \in [B_\zeta]^{\leq n}$.
  By Claim \ref{1d_claim}, the sequence $\langle u_{a_0{}^\frown \langle \beta \rangle}
  \setminus u_{a_0} \mid \beta \in H_0 \setminus (\sup(B_\zeta) + 1) \rangle$
  consists of pairwise disjoint sets. Since $|u_{a_1}| < \kappa$, it
  follows that, letting $C_{a_0, a_1}$ be the set of $\beta \in H_0 \setminus
  (\sup(B_\zeta) + 1)$ such that $u_{a_0{}^\frown \langle \beta \rangle} \setminus u_{a_0}$
  has nonempty intersection with $u_{a_1}$, we have $|C_{a_0, a_1}|
  < \kappa$. Since the number of such pairs $(a_0, a_1)$ is less than
  $\lambda$, we can find $\beta \in H_0 \setminus
  (\sup(B_\zeta) + 1)$ such that, for all $a_0 \in [B_\zeta]^{<n}$ and
  all $a_1 \in [B_\zeta]^{\leq n}$, we have $\beta \notin C_{a_0, a_1}$.
  Let $\beta_\zeta$ be the least such $\beta$, and continue to the next step of
  the construction.

  To verify that $\langle u_b \mid b \in [H]^n \rangle$
  is a uniform $n$-dimensional $\Delta$-system as witnessed by $\rho$ and
  $\langle \mb{r}_\mb{m} \mid \mb{m} \subseteq n \rangle$, we must show that, for all
  $b_0, b_1 \in [H]^n$, if $b_0$ and $b_1$ are aligned and $\mb{m} =
  \mb{r}(b_0, b_1)$, then $u_{b_0}$ and $u_{b_1}$ are aligned with
  $\mb{r}(u_{b_0}, u_{b_1}) = \mb{r}_\mb{m}$. To this end, fix
  $b_0, b_1 \in [H]^n$ such that $b_0$ and $b_1$ are aligned, and let
  $\mb{m} = \mb{r}(b_0, b_1)$. If $n-1 \in \mb{m}$, then the desired conclusion
  already follows from Claim \ref{top_element_claim}, so assume that
  $n-1 \notin \mb{m}$.

  Without loss of generality, assume that $\max(b_0) < \max(b_1)$.
  By Claim \ref{no_top_element_claim}, we know that
  $u_{b_0}[\mb{r}_{\mb{m}}] = u_{b_1}[\mb{r}_{\mb{m}}]$. It will therefore suffice to show
  that, for all $i < \rho$, if $u_{b_1}(i) \in u_{b_0}$, then $i \in \mb{r}_{\mb{m}}$.

  To this end, fix $i < \rho$ such that $\gamma := u_{b_1}(i) \in u_{b_0}$.
  Let $m^* < n$ be least such that $b_1(m^*) > \max(b_0)$. Notice that
  this $m^*$ exists, since $\max(b_1) > \max(b_0)$.

  \begin{claim} \label{claim_215}
    $\gamma \in u_{b_1[m^*]}$.
  \end{claim}

  \begin{proof}
    We will prove by induction on $\ell \leq n - m^*$ that $\gamma \in
    u_{b_1[n - \ell]}$. First, if $\ell = 0$, then $b_1[n - \ell] = b_1[n] =
    b_1$, and, by assumption, we have $\gamma \in u_{b_1}$.
    Next, suppose that $\ell < n - m^*$ and we have proven that
    $\gamma \in u_{b_1[n-\ell]}$. Then $b_1(n - \ell - 1) > \max(b_0)$, so,
    by our thinning out of $H_0$ to $H$, we know that
    $u_{b_1[n-\ell]} \setminus u_{b_1[n- \ell - 1]}$ is disjoint from
    $u_{b_0}$. Since $\gamma \in u_{b_0}$, it follows
    that $\gamma \in u_{b_1[n - \ell - 1]}$.
  \end{proof}

  \begin{claim} \label{claim_216}
    $\gamma \in u_{b_0[n-1]}$.
  \end{claim}

  \begin{proof}
    Because $b_0$ and $b_1$ are aligned and $\max(b_1) > \max(b_0)$,
    we know that $\max(b_0) \notin b_1$. Since $m^*$ was least with
    $b_1(m^*) > \max(b_0)$, it follows that
    $\max(b_0) > \max(b_1[m^*])$. Therefore, by our thinning out of $H_0$ to
    $H$, we know that $u_{b_0} \setminus u_{b_0[n-1]}$ is disjoint from
    $u_{b_1[m^*]}$. Since $\gamma \in u_{b_1[m^*]}$ by the previous claim,
    it follows that $\gamma \in u_{b_0[n-1]}$.
  \end{proof}

  For $\ell < 2$, let $a_\ell = b_\ell[n-1]$ and $\beta_\ell = b_\ell(n-1)$.
  By the two previous claims and our choice of $\beta_0$ and $\beta_1$, we know that
  \begin{align*}
    \gamma \in u_{b_0[n-1]} \cap u_{b_1[m^*]} &= u_{b_0}[\mb{r}_{n-1}]
    \cap u_{b_1}[\mb{r}_{m^*}]  \\ &\subseteq u_{b_0}[\mb{i}] \cap u_{b_1}[\mb{i}]
    \\ &= u_{a_0{}^\frown \langle \mu_M \rangle}[\mb{i}]
    \cap u_{a_1{}^\frown \langle \mu_M \rangle}[\mb{i}].
  \end{align*}
  In particular, we have $i \in \mb{i}$ and, since
  $u_{b_1}[\mb{i}] = u_{a_1{}^\frown \langle \mu_M \rangle}[\mb{i}]$,
  we also know that $u_{a_1{}^\frown \langle \mu_M \rangle}(i) = \gamma$.

  Since $b_0$ and $b_1$ are aligned, we know that $a_0$ and $a_1$ are aligned,
  and, since $n-1 \notin \mb{m}$, we also have $\mb{r}(a_0, a_1) = \mb{m}$.
  Therefore, by our choice of $H_0$, it follows that
  $u_{a_0{}^\frown \langle \mu_M \rangle}$ and $u_{a_1{}^\frown \langle \mu_M \rangle}$
  are aligned and $\mb{r}(u_{a_0{}^\frown \langle \mu_M \rangle},
  u_{a_1{}^\frown \langle \mu_M \rangle}) = \mb{s}_{\mb{m}}$. Since
  $\gamma \in u_{a_0{}^\frown \langle \mu_M \rangle} \cap
  u_{a_1{}^\frown \langle \mu_M \rangle}$, it follows that $i \in
  \mb{s}_{\mb{m}}$. But since $i \in \mb{i}$ and $\mb{r}_{\mb{m}} =
  \mb{s}_{\mb{m}} \cap \mb{i}$, it follows that $i \in \mb{r}_{\mb{m}}$,
  which finishes the proof of clause (2).

  We finally turn our attention to the ``moreover" clause. To this end,
  fix $m < n$ and $a,b \in [H]^n$ such that $a$ and $b$ are aligned
  above $m$ and $a(m) = b(m)$. Fix $\alpha \in H$ such that $\alpha$ is $m$-possible
  for both $a$ and $b$. We must show that $\tp_{\mathrm{int}}(u_a,u_b) = \tp_{\mathrm{int}}
  (u_{a_{m \mapsto \alpha}}, u_{b_{m \mapsto \alpha}})$.
  Let $a^- = a[n-1]$ and $b^- = b[n-1]$, and let $a^+ := a^-{}^\frown \langle
  \mu_M \rangle$ and $b^+ := b^-{}^\frown \langle \mu_M \rangle$.

  Suppose first that $m = n - 1$, so $a(n-1) = b(n-1)$. By the fact that 
  $A$ satisfies Clause (3) of Claim \ref{a_claim}, we know that
  $\tp(u_a, u_b) = \tp(u_{a^+}, u_{b^+}) = \tp(u_{a_{m \mapsto \alpha}},
  u_{b_{m \mapsto \alpha}})$, and hence $\tp_{\mathrm{int}}(u_a,u_b) = \tp_{\mathrm{int}}
  (u_{a_{m \mapsto \alpha}}, u_{b_{m \mapsto \alpha}})$.

  Suppose next that $m < n - 1$. By the fact that $\langle u_{a ^\frown \langle \mu_M \rangle}
  \mid a \in [H]^{n-1} \rangle$ satisfies the ``moreover" clause in the statement
  of the theorem, we know that
  \begin{align}
    \tp_{\mathrm{int}}(u_{a^+}, u_{b^+}) = \tp_{\mathrm{int}}\left(u_{a^+_{m \mapsto \alpha}},
    u_{b^+_{m \mapsto \alpha}}\right). \tag{$\ast$}
  \end{align}
  Suppose in addition that $a(n-1) = b(n-1)$. Then, again by the fact that $A$ satisfies 
  Clause (3) of Claim \ref{a_claim}, we know that
  \[
    \tp(u_a, u_b) = \tp(u_{a^+}, u_{b^+}) \text{ and }
    \tp\left(u_{a_{m \mapsto \alpha}}, u_{b_{m \mapsto \alpha}}\right) =
    \tp\left(u_{a^+_{m \mapsto \alpha}}, u_{b^+_{m \mapsto \alpha}}\right).
  \]
  Putting this together yields $\tp_{\mathrm{int}}(u_a,u_b) = \tp_{\mathrm{int}}
  (u_{a_{m \mapsto \alpha}}, u_{b_{m \mapsto \alpha}})$, as desired.

  The remaining case is that in which $a(n-1) \neq b(n-1)$. We show that
  $\tp_{\mathrm{int}}(u_a,u_b) \subseteq \tp_{\mathrm{int}}(u_{a_{m \mapsto \alpha}}, u_{b_{m
  \mapsto \alpha}})$. A symmetric argument will yield
  the reverse inclusion. To this end, fix $(i,j) \in \tp_{\mathrm{int}}(u_a,u_b)$.
  Thus, we have $u_a(i) = u_b(j) = \gamma$ for some ordinal $\gamma$.
  Since $a$ and $b$ are aligned above $m$, $a(m) = b(m)$, and $a(n-1)
  \neq b(n-1)$, it follows that $b(n-1) \notin a$ and $a(n-1) \notin b$. An
  argument exactly as in the proofs of Claims \ref{claim_215} and \ref{claim_216}
  then shows that $\gamma \in u_{a^-} \cap u_{b^-} = u_a[\mb{i}] \cap
  u_b[\mb{i}]$, and hence we have $i,j \in \mb{i}$.

  Since $i,j \in \mb{i}$, the fact that $A$ satisfies Clause (1c) of Claim \ref{a_claim} 
  implies that $u_a(i) = u_{a^+}(i)$ and
  $u_b(j) = u_{b^+}(j)$, and hence $(i,j) \in \tp_{\mathrm{int}}(u_{a^+}, u_{b^+})$.
  By equation $(\ast)$ above, we have $(i,j) \in \tp_{\mathrm{int}}\left(u_{a^+_{m
  \mapsto \alpha}}, u_{b^+_{m \mapsto \alpha}}\right)$. Again
  by the facts that $A$ satisfies Clause (1c) of Claim \ref{a_claim} 
  and that $i,j \in \mb{i}$, we have
  $u_{a_{m \mapsto \alpha}}(i) = u_{a^+_{m \mapsto \alpha}}(i)$ and
  $u_{b_{m \mapsto \alpha}}(j) = u_{b^+_{m \mapsto \alpha}}(j)$, so
  $(i,j) \in \tp_{\mathrm{int}}(u_{a_{m \mapsto \alpha}}, u_{b_{m
  \mapsto \alpha}})$, thus finishing the proof.
\end{proof}

The following corollary gives an important special case, obtained from setting
$\kappa = \aleph_0$ and $\lambda = \aleph_1$ in Theorem \ref{general_theorem}.

\begin{corollary} \label{aleph_1_cor}
  Suppose that $1 \leq n < \omega$, and let $\mu := \beth_{n-1}^+$.
  If $\left\langle u_b \mid b \in [\mu]^n \right\rangle$ is a family of finite
  sets of ordinals and $g:[\mu]^n \rightarrow \omega$ is a function, then there is
  $H \in [\mu]^{\aleph_1}$ such that $\left\langle u_b \mid b \in [H]^n \right\rangle$
  is a uniform $n$-dimensional $\Delta$-system and $g \restriction [H]^n$ is
  constant. \qed
\end{corollary}

We end this section with a discussion of the optimality of Theorem \ref{general_theorem}.
It can be argued that, if $\kappa < \lambda \leq \mu$ are infinite cardinals,
$1 \leq n < \omega$, and $\mu \rightarrow (\lambda)^{2n}_{2^{<\kappa}}$, then
any sequence $\langle u_a \mid a \in [\mu]^n \rangle$ consisting of elements of
$[\mathrm{On}]^{<\kappa}$ can be thinned out to a uniform $n$-dimensional
$\Delta$-system of size $\lambda$ (see \cite{higher_limits} for such an argument).

In general, $\mu \rightarrow (\lambda)^{2n}_{2^{<\kappa}}$ is a stronger
assertion than $\mu \geq \sigma(\lambda, n)$, which is our assumption in
Theorem \ref{general_theorem}, so this argument yields weaker results than
those of Theorem \ref{general_theorem}. However, if $\lambda$ is weakly compact,
then we have $\lambda \rightarrow (\lambda)^{2n}_{2^{<\kappa}}$ for all
$1 \leq n < \omega$ and all $\kappa < \lambda$, so we obtain the following
corollary.

\begin{corollary} \label{wkly_compact_cor}
  Suppose that $1 \leq n < \omega$ and that $\kappa < \lambda$ are infinite
  cardinals, with $\lambda$ being weakly compact. Suppose also that
  $\langle u_a \mid a \in [\lambda]^n \rangle$ is a sequence consisting of
  elements of $[\mathrm{On}]^{<\kappa}$. Then there is $H \in [\lambda]^\lambda$
  such that $\langle u_a \mid a \in [H]^n \rangle$ is a uniform
  $n$-dimensional $\Delta$-system. \qed
\end{corollary}

If $\lambda$ is \emph{not} weakly compact, though, then our result is optimal in
the sense that the value of $\mu$ cannot be decreased. This is true
even disregarding clause (1) or the ``moreover clause" of Theorem \ref{general_theorem}
and focusing only on the higher-dimensional $\Delta$-systems (and not even
requiring that the $\Delta$-systems be uniform), for essentially the same reason that
the Erd\H{o}s-Rado theorem is optimal.

\begin{proposition} \label{optimal_prop_1}
  Suppose that $1 \leq n < \omega$ and $\lambda$ is a regular uncountable cardinal that is
  not weakly compact, and suppose that $\mu < \sigma(\lambda, n)$. Then there is
  a sequence $\langle u_a \mid a \in [\mu]^n \rangle$ consisting of finite sets
  of ordinals such that there is no $H \in [\mu]^\lambda$ for which
  $\langle u_a \mid a \in [H]^n \rangle$ is an $n$-dimensional
  $\Delta$-system.
\end{proposition}

\begin{proof}
  If $n = 1$, then we have $\mu < \lambda$, so the result is trivial. So
  assume that $n > 1$. Since $\lambda$ is uncountable, regular, and not weakly compact,
  \cite[Corollary 21.5]{ehmr} implies that $2^{<\lambda} \not\rightarrow (\lambda)^2_2$.
  Therefore, by successive applications of \cite[Lemma 5A]{ehr}, which is the lemma
  establishing the optimality of the Erd\H{o}s-Rado theorem, we have, for all
  $m < \omega$, $\beth_m(2^{<\lambda}) \not\rightarrow (\lambda)^{2 + m}_2$.
  By Remark \ref{sigma_remark}(2), $\sigma(\lambda, n) =
  \left(\beth_{n - 2}(2^{<\lambda})\right)^+$. Therefore, we have $\mu \leq \beth_{n-2}(2^{<\lambda})$,
  so there is a function $c:[\mu]^n \rightarrow 2$ that is not constant on $[H]^n$ for
  any $H \in [\mu]^\lambda$. For each $a \in [\mu]^n$, simply let $u_a := c(a)$.
  Now suppose that $H \in [\mu]^\lambda$, and suppose for sake of contradiction
  that $\langle u_a \mid a \in [H]^n \rangle$ is an $n$-dimensional $\Delta$-system,
  as witnessed by roots $\langle R^\mb{m}_a \mid \mb{m} \subseteq n, ~ a \in [H]^{|\mb{m}|}
  \rangle$. Using the fact that $c$ is not constant on $[H']^n$ for any
  unbounded $H' \subseteq H$, we can fix three sets $a_0 < a_1 < a_2$ in $[H]^n$
  such that $c(a_0) = 0$ and $c(a_1) = c(a_2) = 1$. By the definition of an
  $n$-dimensional $\Delta$-system, we should have $u_{a_0} \cap u_{a_1} =
  R^\emptyset_\emptyset = u_{a_1} \cap u_{a_2}$. However, we actually have
  $u_{a_0} \cap u_{a_1} = \emptyset$ and $u_{a_1} \cap u_{a_2} = 1$, which
  is our desired contradiction.
\end{proof}

Before turning to the optimality of the value of $\kappa$ in Theorem \ref{general_theorem},
we pause to summarize the results of this section thus far in a corollary
connecting Theorem \ref{general_theorem} and Proposition \ref{optimal_prop_1}
with the Erd\H{o}s-Rado theorem.

\begin{corollary}\label{general_cor_1}
  Suppose that $1 \leq n < \omega$ and that $\lambda$ and $\mu$ are infinite
  regular cardinals such that $\lambda$ is uncountable but not weakly compact.
  Then the following are equivalent:
  \begin{enumerate}
    \item $\mu \geq \sigma(\lambda, n)$;
    \item $\mu \rightarrow (\lambda)^n_2$;
    \item $\mu \rightarrow (\lambda + (n-1))^n_\nu$ for every $\nu < \lambda$;
    \item for every sequence $\langle u_b \mid b \in [\mu]^n \rangle$ such
    that each $u_b$ is a finite set, there is $H \in [\mu]^\lambda$ such that
    $\langle u_b \mid b \in [H]^n \rangle$ is an $n$-dimensional $\Delta$-system;
    \item the conclusion of Theorem \ref{general_theorem} holds for $n$, $\lambda$,
    and $\mu$, and for any choice of $\kappa$, $\nu$, $g:[\mu]^n \rightarrow \nu$,
    and $\langle u_b \mid b \in [\mu]^n \rangle$ such that
    \begin{enumerate}
      \item $\nu < \lambda$;
      \item $\lambda$ is ${<}\kappa$-inaccessible; and
      \item $u_b \in [\mathrm{On}]^{<\kappa}$ for every $b \in [\mu]^n$.
    \end{enumerate}
  \end{enumerate}
\end{corollary}

\begin{proof}
  $(1) \Rightarrow (3)$ is the Erd\H{o}s-Rado theorem, or the pigeonhole principle
  if $n = 1$ (it can also be extracted
  from our proof of Theorem \ref{general_theorem}), and $(3) \Rightarrow (2)$
  is immediate. $(1) \Rightarrow (5)$ is Theorem \ref{general_theorem}, and
  $(5) \Rightarrow (4)$ follows by setting $\kappa = \aleph_0$ in Theorem \ref{general_theorem} and invoking Proposition \ref{independence_prop}(2).
  $(4) \Rightarrow (2)$ is precisely the second half of the proof of Proposition
  \ref{optimal_prop_1}. Finally, $(2) \Rightarrow (1)$ follows from the optimality
  of the Erd\H{o}s-Rado theorem (the argument in the first half of the proof of
  Proposition \ref{optimal_prop_1}).
\end{proof}

We now turn to the optimality of $\kappa$ in Theorem \ref{general_theorem};
in other words, we investigate the necessity of the requirement that $\lambda$
be ${<}\kappa$-inaccessible in the statement of the theorem.
It turns out that the optimality of $\kappa$ is slightly more
complicated than the optimality of $\mu$, since even if $\lambda$ is not
${<}\kappa$-inaccessible, it could
be the case that $\sigma(\lambda, n) = \sigma(\lambda^*,n)$ for some $\lambda^*
> \lambda$ such that $\lambda^*$ \emph{is} ${<}\kappa$-inaccessible. For example,
suppose that $2^{\aleph_0} = \aleph_2$ and $2^{\aleph_1} = 2^{\aleph_2} = \aleph_3$.
Then $\sigma(\aleph_2, n) = \sigma(\aleph_3, n)$ for all $n \geq 2$. Also,
$\aleph_3$ is ${<}\aleph_1$-inaccessible, so Theorem \ref{general_theorem}
holds for $\lambda = \aleph_3$ and $\kappa = \nu = \aleph_1$ (and any value of $n$).
This immediately implies that the conclusion of Theorem \ref{general_theorem} holds for
$\lambda = \aleph_2$, $\kappa = \nu = \aleph_1$, and $2 \leq n < \omega$,
despite the fact that $\aleph_2$ is not ${<}\aleph_1$-inaccessible. We can show
however, that this is essentially the only way in which the value of $\kappa$
in Theorem \ref{general_theorem} can fail to be optimal.

\begin{proposition} \label{optimal_prop_2}
  Suppose that $1 \leq n < \omega$ and $\kappa < \lambda$ are infinite cardinals
  such that $\lambda$ is regular and not ${<}\kappa$-inaccessible. Let
  $\lambda^* = (\lambda^{<\kappa})^+$, and suppose that $\mu < \sigma(\lambda^*, n)$.
  Then there is a sequence $\langle u_a \mid a \in [\mu]^n \rangle$ consisting
  of elements of $[\lambda]^{<\kappa}$ such that there is no $H \in [\mu]^\lambda$
  for which $\langle u_a \mid a \in [H]^n \rangle$ is an $n$-dimensional
  $\Delta$-system.
\end{proposition}

\begin{proof}
  Fix a cardinal $\nu < \lambda$ such that $\nu \geq \kappa$ and
  $\nu^{<\kappa} \geq \lambda$.
  Next, fix an injective sequence $\langle x_\eta \mid \eta < \lambda \rangle$
  of elements of $[\nu]^{<\kappa}$ such that, for all distinct $\eta, \xi
  < \lambda$, neither of $x_\eta$ nor $x_\xi$ is a subset of the other.
  One way to see that this can be done is the following. Let
  $\langle \kappa_i \mid i < \theta \rangle$ be such that
  \begin{itemize}
    \item if $\kappa$ is a successor cardinal, then $\theta = 1$ and
    $\kappa_0$ is its immediate predecessor (so $\nu^{<\kappa} = \nu^{\kappa_0}$); or
    \item if $\kappa$ is a limit cardinal, then $\theta = \cf(\kappa)$ and
    $\langle \kappa_i \mid i < \theta \rangle$ is a strictly increasing sequence of
    cardinals converging to $\kappa$.
  \end{itemize}
  Now let $\langle f_\eta \mid \eta < \lambda \rangle$ be an injective
  sequence of elements of $\bigcup_{i < \theta} {^{\kappa_i}}\nu$.
  Partition $\nu$ into pairwise
  disjoint pieces $\langle A_i \mid i < \theta \rangle$,
  each of size $\nu$ and, for each $i < \theta$, let
  $\pi_i : \kappa_i \times \nu \rightarrow A_i$ be a bijection.
  Now, viewing elements of ${^{\kappa_i}}\nu$ as subsets of
  $\kappa_i \times \nu$, for each $\eta < \lambda$, let $i_\eta$ be the
  unique $i < \theta$ such that $f_\eta \in {^{\kappa_i}}\nu$, and let
  $x_\eta := \pi_{i_\eta}``f_\eta$. Then $\langle x_\eta \mid \eta < \lambda \rangle$
  is as desired. Similarly, fix an injective sequence
  $\langle y_\alpha \mid \alpha < \lambda^{<\kappa} \rangle$ of elements of
  $[\lambda \setminus \nu]^{<\kappa}$ such that, for all $\alpha < \beta < \lambda^{<\kappa}$, neither of 
  $y_\alpha$ nor $y_\beta$ is a subset of the other.
  For all $\alpha < \lambda^{<\kappa}$, let $\eta_\alpha = \sup(y_\alpha)$.
  Since $\lambda$ is regular, we have $\eta_\alpha < \lambda$.

  Suppose first that $n = 1$. Then $\sigma(\lambda^*, 1) = \lambda^* =
  (\lambda^{<\kappa})^+$, so we can assume that $\mu = \lambda^{<\kappa}$. For all
  $\alpha < \mu$, let $u_\alpha := y_\alpha \cup x_{\eta_\alpha}$. Fix $H \in
  [\mu]^\lambda$, and suppose for sake of contradiction that $\langle u_\alpha
  \mid \alpha \in H \rangle$ is a $\Delta$-system, with root $r$. Let
  $r^- := r \cap \nu$ and $r^+ := r \setminus \nu$. Note that, for all
  distinct $\alpha, \beta \in H$, we have $y_\alpha \cap y_\beta = r^+$
  and $x_{\eta_\alpha} \cap x_{\eta_\beta} = r^-$.

  There are now two cases to consider, depending on whether or not
  $\{\eta_\alpha \mid \alpha \in H\}$ is unbounded in $\lambda$. Suppose first that
  $\eta^* := \sup\{\eta_\alpha \mid \alpha \in H \}$ is less than $\lambda$.
  Then $\langle y_\alpha \setminus r^+ \mid \alpha \in H \rangle$ is an injective
  sequence of pairwise disjoint nonempty subsets of $\eta^* + 1$, contradicting the
  fact that $|H| = \lambda > \eta^* + 1$. Suppose next that $\eta^* = \lambda$.
  Then $\{x_{\eta_\alpha} \setminus r^-
  \mid \alpha \in H\}$ is a set of size $\lambda$ consisting of pairwise disjoint nonempty
  subsets of $\nu$, contradicting the fact that $\lambda > \nu$.

  Now suppose that $n > 1$. Then, by Remark \ref{sigma_remark}(1),
  we know that $\sigma(\lambda^*, n) = (\beth_{n-1}(\lambda^{<\kappa}))^+$, so we
  can assume that $\mu = \beth_{n-1}(\lambda^{<\kappa})$. For any infinite cardinal $\chi$,
  the coloring $d:[{^\chi}2]^2 \rightarrow \chi$ defined by letting
  $d(f,g)$ be the least $\xi < \chi$ for which $f(\xi) \neq g(\xi)$ for all
  distinct $f,g \in {^\chi}2$ witnesses the negative partition relation
  $2^\chi \not\rightarrow (3)^2_\chi$. Therefore, setting $\chi = \lambda^{<\kappa}$
  and repeatedly applying \cite[Lemma 5A]{ehr}, we have $\mu
  \not\rightarrow (\aleph_0)^n_{\lambda^{<\kappa}}$. Let
  $c:[\mu]^n \rightarrow \lambda^{<\kappa}$ witness this negative partition relation.

  Now define $\langle u_a \mid a \in [\mu]^n \rangle$ by letting $u_a :=
  y_{c(a)} \cup x_{\eta_{c(a)}}$ for all $a \in [\mu]^n$. Fix $H \in [\mu]^\lambda$,
  and suppose for sake of contradiction that $\langle u_a \mid a \in [H]^n \rangle$
  is an $n$-dimensional $\Delta$-system, as witnessed by roots $\langle R^{\mb{m}}_a
  \mid \mb{m} \subseteq n, ~ a \in [H]^{|\mb{m}|}\rangle$. Let $r := R^\emptyset_\emptyset$,
  $r^- := r \cap \nu$, and $r^+ := r \setminus \nu$. Since $c$ witnesses
  $\mu \not\rightarrow (\aleph_0)^n_{\lambda^{<\kappa}}$, we can find disjoint sets
  $a_0, a_1 \in [H]^n$ such that $c(a_0) \neq c(a_1)$. Now arbitrarily fix
  a set $a_\gamma \in [H]^n$ for each $2 \leq \gamma < \lambda$ in such a way that
  $\langle a_\gamma \mid \gamma < \lambda \rangle$ is an injective sequence of
  pairwise disjoint sets. By the definition of $\langle u_a \mid a \in [H]^n \rangle$ and 
  our choice of $r$, $r^+$, and $r^-$, we know that, for all $\gamma < \delta < \lambda$,
  we have $u_{a_\gamma} \cap u_{a_\delta} = r$, and hence
  $y_{c(a_\gamma)} \cap y_{c(a_\delta)} = r^+$ and
  $x_{\eta_{c(a_\gamma)}} \cap x_{\eta_{c(a_\delta)}} = r^-$.

  There are now two possibilities. First, suppose that there are $\gamma < \delta
  < \lambda$ for which $u_{a_\gamma} = u_{a_\delta}$. Then we can find
  $\ell < 2$ for which $u_{a_\ell} \neq u_{a_\gamma}$ (and hence
  $u_{a_\gamma} \not\subseteq u_{a_\ell}$). But now
  we are in the same situation as in the proof of Proposition \ref{optimal_prop_1}:
  we must have $u_{a_\gamma} \cap u_{a_\delta} = r =
  u_{a_\ell} \cap u_{a_\gamma}$, but $u_{a_\gamma} \cap u_{a_\delta} = u_{a_\gamma}$,
  and since $u_{a_\gamma} \not\subseteq u_{a_\ell}$, we have
  $u_{a_\ell} \cap u_{a_\gamma} \neq u_{a_\gamma}$, which is a contradiction.

  The other possibility is that the sets $\langle u_{a_\gamma} \mid \gamma < \lambda
  \rangle$ are all pairwise disjoint. There are now two subcases, depending on
  whether or not $\eta^* := \sup\{\eta_{c(a_\gamma)} \mid \gamma < \lambda\}$ is
  equal to $\lambda$. If $\eta^* < \lambda$, then $\langle u_{a_\gamma} \setminus
  r \mid \gamma < \lambda \rangle$ is an injective sequence of pairwise disjoint nonempty
  subsets of $\max\{\eta^* + 1, \nu\}$. If $\eta^* = \lambda$, then
  $\{x_{\eta_{c(a_\gamma)}} \setminus r^- \mid \gamma < \lambda\}$ is a set of
  size $\lambda$ consisting of pairwise disjoint nonempty subsets of $\nu$. In either case,
  we contradict the fact that $\lambda > \max\{\eta^* + 1, \nu\}$.
\end{proof}

If $\kappa < \lambda$ are both regular infinite cardinals and $\lambda$ is
not ${<}\kappa$-inaccessible, then $(\lambda^{<\kappa})^+$ is the least
${<}\kappa$-inaccessible cardinal greater than or equal to $\lambda$ (it can fail to be
${<}\kappa$-inaccessible if $\kappa$ is singular). We can therefore combine
the results of this section in the following equivalence.

\begin{corollary} \label{general_cor_2}
  Suppose that $1 \leq n < \omega$ and $\kappa < \lambda$ are infinite regular cardinals.
  Let $\lambda^*$ be the least ${<}\kappa$-inaccessible cardinal greater than
  or equal to $\lambda$, and suppose that $\mu$ is an infinite cardinal. Then
  the following are equivalent.
  \begin{enumerate}
    \item $\mu \geq \sigma(\lambda^*, n)$;
    \item the conclusion of Theorem \ref{general_theorem} holds for
    $n$, $\kappa$, $\lambda$, and $\mu$ with any choice of $\nu < \lambda$,
    $g:[\mu]^n \rightarrow \nu$, and $\langle u_b \mid b \in [\mu]^n \rangle$
    with each $u_b$ in $[\mathrm{On}]^{<\kappa}$;
    \item for every sequence $\langle u_a \mid a \in [\mu]^n \rangle$ such that
    each $u_a$ is a set of cardinality less than $\kappa$, there is $H \in
    [\mu]^\lambda$ such that $\langle u_a \mid a \in [H]^n \rangle$ is an
    $n$-dimensional $\Delta$-system.
  \end{enumerate}
\end{corollary}

\begin{proof}
  $(1) \Rightarrow (2)$ follows from Theorem \ref{general_theorem}, and
  $(2) \Rightarrow (3)$ is immediate. If $\lambda$ is ${<}\kappa$-inaccessible,
  then $\lambda^* = \lambda$, in which case $(3) \Rightarrow (1)$ follows from
  Proposition \ref{optimal_prop_1}. If $\lambda$ is not ${<}\kappa$-inaccessible,
  then $(3) \Rightarrow (1)$ follows from Proposition \ref{optimal_prop_2} and
  the observation that $\lambda^* = (\lambda^{<\kappa})^+$ in this case.
\end{proof}

\section{Chain conditions} \label{chain_condition_section}

One of the primary uses of the classical $\Delta$-system lemma is in proving that certain
forcing notions satisfy chain conditions. For example, one of the first applications
that many people learn is in the proof that the forcing
notion to add any number of Cohen reals is $\kappa$-Knaster for every regular
uncountable $\kappa$:

\begin{lemma} \label{knaster_lemma}
  Let $\chi$ be any infinite cardinal, and let $\bb{P} = \mathrm{Add}(\omega, \chi)$
  be the forcing to add $\chi$-many Cohen reals. Suppose that $\kappa$ is a
  regular uncountable cardinal and $\langle p_\alpha \mid \alpha < \kappa \rangle$ is a
  sequence of conditions from $\bb{P}$. Then there is an unbounded $A \subseteq
  \kappa$ such that $\langle p_\alpha \mid \alpha \in A \rangle$ consists of
  pairwise compatible conditions.
\end{lemma}

During forcing constructions involving higher-dimensional combinatorial statements,
one frequently encounters sequences of conditions indexed not by single ordinals
but by $n$-element sets of ordinals for some $n > 1$. One would then like to
find a large set such that the restriction of the sequence to that set satisfies
certain uniformities analogous to the uniformities exhibited by
$\langle p_\alpha \mid \alpha \in A \rangle$ in Lemma \ref{knaster_lemma}.
A first, na\"{i}ve attempt at formulating a statement to this effect, similar to
our overly optimistic first attempt to define higher-dimensional $\Delta$-systems
at the start of Section \ref{def_section}, might look vaguely as follows:

\begin{quote}
  Let $\chi$ be an infinite cardinal and $1 \leq n < \omega$, and let $\bb{P}$ be the forcing to add
  $\chi$-many Cohen reals. Then there is a sufficiently large
  regular cardinal $\mu \leq \chi$ such that, for every sequence
  $\langle p_{a} \mid a \in [\mu]^n \rangle$ of conditions in $\bb{P}$, there is a ``large"
  set $H \subseteq \mu$ such that $\langle p_a \mid a \in [H]^n \rangle$
  consists of pairwise compatible conditions.
\end{quote}

It is easily seen that such a statement cannot possibly hold if $n > 1$, however. Indeed
suppose that $n = 2$ and, for all $(\alpha, \beta) \in [\mu]^2$, define a condition
$p_{\alpha\beta} \in \bb{P}$ by letting $\dom(p_{\alpha\beta}) := \{\alpha, \beta\}$,
$p_{\alpha\beta}(\alpha) := 0$, and $p_{\alpha\beta}(\beta) := 1$
(we are thinking of conditions in $\bb{P}$ as
being finite partial functions from $\chi$ to $2$). Then
$p_{\alpha\beta} \perp p_{\beta\gamma}$ for all $(\alpha, \beta, \gamma)
\in [\mu]^3$, so we could not even find a set $H$ of size 3 as in the above
statement. The obvious problem here is that the sets $\{\alpha, \beta\}$ and
$\{\beta, \gamma\}$ are not aligned, and it turns out that this is the only
obstacle. By requiring the compatibility of $p_a$ and $p_b$ only when
$a$ and $b$ are aligned, we obtain a consistent statement. For example:

\begin{lemma} \label{higher_knaster_lemma}
  Suppose that $\lambda$ is a regular uncountable cardinal, $1 \leq n < \omega$, and
  $\mu = \sigma(\lambda, n)$, and suppose that $\bb{P}$ is the forcing notion
  to add $\chi$-many Cohen reals for some infinite cardinal $\chi$. Then, for every sequence
  $\langle p_a \mid a \in [\mu]^n \rangle$ of conditions in $\bb{P}$, there is
  a set $H \in [\mu]^\lambda$ such that, for all $a,b \in [H]^n$, if
  $a$ and $b$ are aligned, then $p_a \parallel p_b$.
\end{lemma}

\begin{proof}
  Fix a sequence $\langle p_a \mid a \in [\mu]^n \rangle$ consisting of conditions
  in $\bb{P}$. For each $a \in [\mu]^n$, let $u_a := \dom(p_a)$ and $k_a := \otp(u_a)$, and let
  $\bar{p}_a : k_a \rightarrow 2$ denote the condition isomorphic to $p_a$, i.e.,
  $\bar{p}_a(i) = p_a(u_a(i))$ for all $i < k_a$. Now apply Theorem
  \ref{general_theorem} to $\langle u_a \mid a \in [\mu]^n \rangle$ and the function
  $a \mapsto \bar{p}_a$ to find an $H \in [\mu]^{\lambda}$, a $k < \omega$, and
  a function $\bar{p} : k \rightarrow 2$ such that
  $\langle u_a \mid a \in [H]^n \rangle$ is a uniform $n$-dimensional
  $\Delta$-system and $\bar{p}_a = \bar{p}$ for all $a \in [H]^n$.

  We claim that $p_a \parallel p_b$ for all aligned $a,b \in [H]^n$. To this end,
  fix $a,b \in [H]^n$ such that $a$ and $b$ are aligned. The only way we could
  have $p_a \perp p_b$ is if there is $\alpha \in u_a \cap u_b$ such that
  $p_a(\alpha) \neq p_b(\alpha)$. Since
  $\langle u_a \mid a \in [H]^n \rangle$ is a uniform $n$-dimensional
  $\Delta$-system, we know that $u_a$ and $u_b$ are aligned. Moreover, we know
  that $\bar{p}_a = \bar{p}_b = \bar{p}$. Therefore, if $\alpha \in
  u_a \cap u_b$, then there is $i < k$ such that $\alpha = u_a(i) =
  u_b(i)$. But then $p_a(\alpha) = \bar{p}(i) = p_b(\alpha)$. Therefore, we have
  $p_a \parallel p_b$.
\end{proof}

\begin{remark}
  If $\lambda$ is weakly compact, then, by Corollary \ref{wkly_compact_cor},
  Lemma \ref{higher_knaster_lemma} still holds with $\mu = \lambda$ rather than
  $\mu = \sigma(\lambda, n)$.
\end{remark}

\section{An application to polarized partition relations} \label{application_section}

In this section, we give a relatively simple application illustrating a typical
use of Theorem \ref{general_theorem} in a forcing argument. The following definition
was introduced by Todor\v{c}evi\'{c}.

\begin{definition}[{\cite[Remark 9.3.3]{todorcevic_book}}]
  Let $1 \leq n < \omega$. Then $\Theta_n$ is the least cardinal $\theta$
  such that, for every function $f : \theta^n \rightarrow \omega$, there is a
  sequence $\langle A_i \mid i < n \rangle$ of infinite subsets of $\theta$
  such that $f \restriction \prod_{i < n} A_i$ is constant.
\end{definition}

We clearly have $\Theta_1 = \aleph_1$. The next proposition establishes lower
bounds for $\Theta_n$ for $n > 1$.

\begin{proposition}
  Suppose that $1 \leq n < \omega$, $\kappa$ is a cardinal, and
  $\Theta_n > \kappa$. Then $\Theta_{n + 1} > \kappa^+$.
\end{proposition}

\begin{proof}
  Since $\Theta_n > \kappa$, we can fix a function $g : \kappa^n \rightarrow
  \omega$ such that $g$ is not constant on any product of $n$ infinite subsets
  of $\kappa$. For each $\beta < \kappa^+$, fix an injective function
  $e_\beta : \beta \rightarrow \kappa$. Then the function $g_\beta :
  \beta^n \rightarrow \omega$ defined by letting
  \[
    g_\beta(\langle \alpha_0, \ldots, \alpha_{n - 1} \rangle) :=
    g(\langle e_\beta(\alpha_0), \ldots, e_\beta(\alpha_{n - 1}) \rangle)
  \]
  for all $\langle \alpha_0, \ldots, \alpha_{n - 1} \rangle \in \beta^n$
  has the property that $g_\beta$ is not constant on any product of
  $n$ infinite subsets of $\beta$.

  We now define a function $f: (\kappa^+)^{n + 1} \rightarrow (n + 2) \times \omega$
  that will not be constant on any product of $(n + 1)$ infinite subsets of
  $\kappa^+$. This can easily be coded as a function into $\omega$, so this
  suffices to prove the proposition.

  Given $\vec{\alpha} = \langle \alpha_0, \ldots, \alpha_n \rangle \in (\kappa^+)^{n+1}$
  and $i \leq n$, let $\vec{\alpha}^i$ denote the sequence formed by removing
  $\alpha_i$ from $\vec{\alpha}$, i.e., $\vec{\alpha}^i := \langle \alpha_0,
  \ldots, \alpha_{i - 1}, \alpha_{i + 1}, \ldots, \alpha_n \rangle$. Let us
  now define $f(\vec{\alpha})$. If there are $i < j \leq n$ such that
  $\alpha_i = \alpha_j$, then let $f(\vec{\alpha}) := (n + 1, 0)$.
  Otherwise, let $i \leq n$ be such that $\alpha_j < \alpha_i$ for all
  $j \in (n + 1) \setminus \{i\}$, and let $f(\vec{\alpha}) := (i,
  g_{\alpha_i}(\vec{\alpha}^i))$.

  Suppose for sake of contradiction that $\langle A_i \mid i \leq n \rangle$
  is a sequence of infinite subsets of $\kappa^+$ such that
  $f \restriction \prod_{i \leq n} A_i$ is constant, taking value
  $(m, k)$. First note that we can always find a sequence
  $\vec{\alpha} \in \prod_{i \leq n} A_i$ whose coordinates are all distinct,
  so it cannot be the case that $m = n + 1$. Thus, $m \leq n$, so, by our
  definition of $f$, it follows that $A_j < A_m$ for all $j \in (n + 1) \setminus
  \{m\}$. Fix $\beta \in A_m$, and define $\langle A^*_j \mid j < n \rangle$
  by letting $A^*_j := A_j$ for $j < m$ and $A^*_j := A_{j + 1}$ for
  $m \leq j < n$. Then each $A^*_j$ is an infinite subset of $\beta$ and,
  by our definition of $f$, it follows that $g_\beta \restriction
  \prod_{j < n} A^*_j$ is constant, taking value $k$, contradicting our
  assumptions about $g_\beta$.
\end{proof}

In particular, we immediately obtain the following corollary, answering a part of
Question 9.3.4 from \cite{todorcevic_book}.

\begin{corollary} \label{theta_cor_1}
  $\Theta_n \geq \aleph_n$ for all $1 \leq n < \omega$. \qed
\end{corollary}

It follows easily from the Erd\H{o}s-Rado theorem that $\Theta_n \leq
\beth_{n - 1}^+$ for all $1 \leq n < \omega$. In particular, if
$\mathrm{GCH}$ holds, then $\Theta_n = \aleph_n$ for all $1 \leq n < \omega$.
We now apply Theorem \ref{general_theorem} to prove that adding any number of
Cohen reals preserves the inequality $\Theta_n \leq (\beth_{n-1}^+)^V$.
In fact, we will prove that a slightly stronger partition relation,
which easily implies $\Theta_n \leq (\beth_{n-1}^+)^V$, holds after forcing
to add the Cohen reals.

\begin{theorem} \label{polarized_forcing_theorem}
  Suppose that $1 \leq n < \omega$ and $\chi$ is an infinite cardinal. Let $\mu =
  \beth_{n-1}^+$, and let $\bb{P}$ be the forcing to add $\chi$-many Cohen reals.
  Then the following statement holds in $V^{\bb{P}}$:
  \begin{quote}
    For every function $c:[\mu]^n \rightarrow \omega$, there is a sequence
    $\langle A_m \mid m < n \rangle$ such that
    \begin{itemize}
      \item for all $m < n$, $A_m$ is a subset of $\mu$ of order type $\omega + 1$;
      \item for all $m < m' < n$, we have $A_m < A_{m'}$;
      \item $c \restriction \prod_{m < n} A_m$ is constant.
    \end{itemize}
  \end{quote}
\end{theorem}

\begin{proof}
  We think of conditions in $\bb{P}$ as being finite partial functions from
  $\chi$ to $2$, ordered by reverse inclusion. Given a condition $p \in \bb{P}$,
  let $\bar{p}$ denote the function from $|\dom(p)|$ to $2$ defined by letting
  $\bar{p}(i) := p(\dom(p)(i))$ for all $i < |\dom(p)|$.

  Since the conclusion of the theorem is trivial if
  $n = 1$, we may assume that $n > 1$. Fix a condition $p \in \bb{P}$ and
  a $\bb{P}$-name $\dot{c}$ forced by $p$ to be a function from
  $[\mu]^n$ to $\omega$. For each $b \in [\mu]^n$, find a condition
  $q_b \leq p$ and a color $k_b < \omega$ such that $q_b \Vdash ``\dot{c}(b)
  = k_b"$. Let $u_b := \dom(q_b)$, and define a function $g : [\mu]^n \rightarrow
  {^{<\omega}}2 \times \omega$ by letting $g(b) := \langle \bar{q}_b, k_b \rangle$
  for all $b \in [\mu]^n$. Apply Theorem \ref{general_theorem} to find
  $H \in [\mu]^{\aleph_1}$ such that $\langle u_b \mid b \in [H]^n \rangle$
  is a uniform $n$-dimensional $\Delta$-system and $g \restriction [H]^n$
  is constant, taking value $\langle \bar{q}, k \rangle$. By taking an initial
  segment if necessary, assume that we in fact have $\otp(H) = \omega_1$.
  Note that, if $b$ and $b'$ are aligned elements of $[H]^n$, then
  $q_b$ and $q_{b'}$ are compatible in $\bb{P}$.

  Let $\rho := |\bar{q}|$, and let $\langle \mb{r}_\mb{m} \subseteq \rho \mid
  \mb{m} \subseteq n \rangle$ witness the fact that $\langle u_b \mid
  b \in [H]^n \rangle$ is a uniform $n$-dimensional $\Delta$-system.
  For each $m < n$ and each $a \in [H]^m$, define $u_a$ by letting
  $b$ be any element of $[H]^n$ such that $b[m] = a$ and then letting
  $u_a := u_b[\mb{r}_m]$ (we are thinking of $m$ as an initial subset of $n$ here).
  Then set $q_a := q_b \restriction u_a$. By Proposition \ref{independence_prop}
  and the fact that $\bar{q}_b = \bar{q}$ for all $b \in [H]^n$, it follows that our
  definition of $u_a$ and $q_a$ is independent of our choice of $b$.

  By the arguments of Claim \ref{1d_claim}, we know that, for every $m < n$
  and every $a \in [H]^m$, the sequence $\langle u_{a ^\frown \langle \beta
  \rangle} \mid \beta \in H \setminus (\max(a) + 1) \rangle$ is a 1-dimensional
  $\Delta$-system, with root $u_a$. Since $q_b \leq p$ for all
  $b \in [H]^n$, it follows that $\dom(p) \subseteq u_\emptyset$ and
  $q_\emptyset \leq p$. We will show that $q_\emptyset$ forces the
  existence of a sequence $\langle A_m \mid m < n \rangle$ in $V^\bb{P}$
  such that
  \begin{itemize}
    \item each $A_m$ is a subset of $\mu$ of order type $\omega + 1$;
    \item $A_{m} < A_{m'}$ for all $m < m' < n$;
    \item the realization of $\dot{c}$ is constant when restricted to
    $\prod_{m < n} A_m$, with value $k$.
  \end{itemize}
  Since $p$ was arbitrary, this suffices to prove the theorem. We first need
  the following claim.

  \begin{claim} \label{unbounded_extension_claim}
    Suppose that $m < n$, $a \in [H]^m$, and $\gamma \in H \setminus (\max(a) + 1)$.
    Then the set $D_{a, \gamma} := \{q_{a ^\frown \langle \beta \rangle} \mid \beta \in H
    \setminus \gamma \}$ is predense below $q_a$ in $\bb{P}$.
  \end{claim}

  \begin{proof}
    Fix a condition $r \leq q_a$. We will find an element of $D_{a, \gamma}$
    compatible with $r$. Since $\langle u_{a ^\frown \langle \beta \rangle}
    \mid \beta \in H \setminus \gamma \rangle$ is an infinite 1-dimensional
    $\Delta$-system with root $u_a$, and since $\dom(r)$ is finite, we
    can find $\beta \in H \setminus \gamma$ such that $u_{a ^\frown \langle \beta
    \rangle} \setminus u_a$ is disjoint from $\dom(r)$. But then
    $q_{a ^\frown \langle \beta \rangle} \restriction \dom(r) = q_a$, so,
    since $r \leq q_a$, it follows that $r \cup q_{a ^\frown \langle \beta \rangle}$
    is a condition in $\bb{P}$, so $q_{a ^\frown \langle \beta \rangle}$ is an
    element of $D_{a, \gamma}$ compatible with $r$.
  \end{proof}

  Now suppose that $G$ is $\bb{P}$-generic over $V$ with $q_\emptyset \in G$, and
  let $c$ be the realization of $\dot{c}$ in $V[G]$.
  By applying Claim \ref{unbounded_extension_claim} $n$ times, working in
  $V[G]$, we can recursively choose an increasing sequence $\langle \delta_m \mid
  m < n \rangle$ of elements of $H$ such that, letting $d = \{\delta_m \mid m
  < n \}$, we have
  \begin{itemize}
    \item $q_d \in G$;
    \item $H \cap \delta_0$ is infinite;
    \item for all $m < n - 1$, $H \cap (\delta_{m + 1} \setminus
    (\delta_m + 1))$ is infinite.
  \end{itemize}
  Let $A^*_0$ denote the set of the first $\omega$-many elements of $H \cap \delta_0$ and,
  for all $m < n - 1$, let $A^*_{m + 1}$ denote the set of the first
  $\omega$-many elements of $H \cap (\delta_{m + 1} \setminus
  (\delta_m + 1))$.

  We now construct an $n \times \omega$ matrix $\langle \alpha_{m, \ell} \mid m < n,
  ~ \ell < \omega \rangle$ such that
  \begin{itemize}
    \item for all $m < n$, $\langle \alpha_{m, \ell} \mid \ell < \omega \rangle$
    is an increasing sequence of elements of $A^*_m$;
    \item letting $A_m = \{\alpha_{m, \ell} \mid \ell < \omega\} \cup \{\delta_m\}$
    for each $m < n$, we have $q_b \in G$ for all $b \in \prod_{m < n} A_m$.
  \end{itemize}
  The construction is by recursion on the anti-lexicographical order on
  $n \times \omega$, i.e., we set $(m, \ell) < (m', \ell')$ if $\ell < \ell'$ or
  ($\ell = \ell'$ and $m < m'$). During the construction, at stage $(m, \ell)$,
  for all $m' < n$, we will let $A_{m'} \restriction (m, \ell)$ denote the
  set $\left\{\alpha_{m', \ell'} \mid \ell' \leq \ell \right\} \cup \{\delta_{m'}\}$
  if $m' < m$ and $\left\{\alpha_{m', \ell'} \mid \ell' < \ell \right\} \cup \{\delta_{m'}\}$
  if $m \leq m'$. In other words,
  $A_{m'} \restriction (m, \ell)$ is simply the portion of $A_{m'}$
  that we have specified before stage $(m, \ell)$ of the construction. Our
  recursion hypothesis will be the assumption that, when we reach
  stage $(m, \ell)$, for all $b \in \prod_{m' < n} A_{m'} \restriction (m, \ell)$,
  we have $q_b \in G$. It will then follow that
  $q^*_{m, \ell} := \bigcup \{q_b \mid b \in \prod_{m' < n} A_{m'} \restriction (m, \ell)
  \}$ is also an element of $G$.

  To begin the construction, note that, for all $m' < n$, we have
  $A_{m'} \restriction (0,0) = \{\delta_{m'}\}$, so $q^*_{0,0} = q_d \in G$.
  Thus, our recursion hypothesis is initially satisfied. Now suppose
  that $(m, \ell) \in n \times \omega$ and we have defined
  $\langle \alpha_{m', \ell'} \mid (m', \ell') < (m, \ell) \rangle$ so that
  the resulting condition $q^*_{m, \ell}$ is in $G$. Temporarily move back to $V$,
  noting that each $A_{m'} \restriction (m, \ell)$ is finite and hence in $V$,
  and $A^*_m$ is also in $V$, as it is definable from $H$ (and $\delta_{m - 1}$,
  if $m > 0$).

  Let $B_0 := \prod_{m' < m} (A_{m'} \restriction (m, \ell))$ and
  $B_1 := \prod_{m < m' < n} (A_{m'} \restriction (m, \ell))$.
  If $\ell > 0$, then let $\gamma := \alpha_{m, \ell - 1} + 1$; if $\ell = 0$,
  then let $\gamma := 0$. For each $\alpha \in A^*_m \setminus \gamma$,
  let $q_\alpha^* := \bigcup\{q_{b_0{} ^\frown \langle \alpha\rangle ^\frown
  b_1} \mid b_0 \in B_0, ~ b_1 \in B_1\}$. Notice that, if $b_0, b'_0 \in B_0$
  and $b_1, b'_1 \in B_1$, then $b_0{}^\frown \langle\alpha\rangle ^\frown b_1$ and
  $b'_0{}^\frown \langle\alpha\rangle ^\frown b'_1$ are aligned, and hence
  $q_{b_0{}^\frown \langle\alpha\rangle ^\frown b_1}$ and $q_{b'_0{}^\frown
  \langle\alpha\rangle ^\frown b'_1}$ are compatible. It follows that $q^*_\alpha$ is a
  condition in $\bb{P}$.

  \begin{claim} \label{matrix_building_claim}
    The set $E := \{q^*_\alpha \mid \alpha \in A^*_m \setminus \gamma\}$ is predense
    below $q^*_{m, \ell}$ in $\bb{P}$.
  \end{claim}

  \begin{proof}
    Fix $r \leq q^*_{m, \ell}$. We will find an element of $E$ compatible with
    $r$. Let $\mb{m} = n \setminus \{m\}$. For each $(b_0, b_1) \in B_0 \times B_1$,
    the sequence $\langle u_{b_0{}^\frown \langle\alpha\rangle ^\frown b_1}
    \mid \alpha \in A^*_m
    \setminus \gamma \rangle$ forms a 1-dimensional $\Delta$-system whose root is
    equal to $u_{b_0{}^\frown \langle\alpha\rangle ^\frown b_1}[\mb{r}_\mb{m}]$ for all
    $\alpha \in A^*_m \setminus \gamma$. Since $A^*_m \setminus \gamma$ is
    infinite and since $\dom(r)$, $B_0$, and $B_1$ are all finite, we can find
    $\alpha \in A^*_m \setminus \gamma$ such that, for all $(b_0, b_1) \in B_0 \times
    B_1$, the set $u_{b_0{}^\frown \langle\alpha\rangle ^\frown b_1} \setminus
    (u_{b_0{}^\frown \langle\alpha\rangle ^\frown b_1}[\mb{r}_{\mb{m}}])$ is disjoint
    from $\dom(r)$.

    We claim that $q_\alpha^*$ and $r$ are compatible. To see this, it suffices to
    show that $q_{b_0{}^\frown \langle\alpha\rangle ^\frown b_1}$ and $r$ are
    compatible for every
    $(b_0, b_1) \in B_0 \times B_1$. Thus, fix $(b_0, b_1) \in B_0 \times B_1$.
    We know that $u_{b_0{}^\frown \langle\alpha\rangle ^\frown b_1} \cap
    \dom(r) \subseteq u_{b_0{}^\frown \langle\alpha\rangle ^\frown b_1}[\mb{r}_{\mb{m}}]$.
    Since $b_0{}^\frown \langle\alpha\rangle ^\frown b_1$ and $b_0{}^\frown \langle\delta_m
    \rangle ^\frown b_1$ are aligned with $\mb{r}(b_0{}^\frown \langle\alpha\rangle ^\frown
    b_1, b_0{}^\frown \langle\delta_m \rangle ^\frown b_1) = \mb{m}$, we also
    know that $q_{b_0{}^\frown \langle\alpha\rangle ^\frown b_1} \parallel q_{b_0{}^\frown
    \langle\delta_m\rangle ^\frown b_1}$
    and $u_{b_0{}^\frown \langle\alpha\rangle ^\frown b_1}[\mb{r}_\mb{m}] =
    u_{b_0{}^\frown \langle\delta_m\rangle ^\frown b_1}[\mb{r}_\mb{m}]$.
    Then
    \[
      q_{b_0{}^\frown \langle\alpha\rangle ^\frown b_1} \restriction
      (u_{b_0{}^\frown \langle\alpha\rangle ^\frown b_1}[\mb{r}_{\mb{m}}]) =
      q_{b_0{}^\frown \langle\delta_m\rangle ^\frown b_1} \restriction
      (u_{b_0{}^\frown \langle\delta_m\rangle ^\frown b_1}[\mb{r}_{\mb{m}}]).
    \]
    But we have $q^*_{m, \ell} \leq q_{b_0{}^\frown \langle\delta_m
    \rangle ^\frown b_1}$, since $b_0{}^\frown \langle\delta_m\rangle ^\frown b_1
    \in \prod_{m' < n}A_{m'} \restriction (m, \ell)$.
    It follows that $r \leq q^*_{m, \ell} \leq q_{b_0{}^\frown \langle\alpha\rangle
    ^\frown b_1} \restriction \dom(r)$. Therefore, $r$ and
    $q_{b_0{}^\frown \langle\alpha\rangle ^\frown b_1}$ are compatible.
  \end{proof}

  Returning to $V[G]$, we can find $\alpha \in A^*_m \setminus \gamma$ such that
  $q^*_\alpha \in G$. But notice that, if we were to set $\alpha_{m, \ell} := \alpha$,
  then, letting $(m, \ell)^+$ denote the anti-lexicographic successor of
  $(m, \ell)$, we would have $q^*_{(m, \ell)^+} = q^*_{m, \ell} \cup q^*_\alpha
  \in G$. We can therefore set $\alpha_{m, \ell} := \alpha$ while maintaining
  the recursion hypothesis, and continue to the next step of the construction.

  At the end of the construction, we have built sets $\langle A_m \mid m < n \rangle$
  such that
  \begin{itemize}
    \item for each $m < n$, $A_m$ is a subset of $H$ and $\otp(A_m) = \omega + 1$;
    \item for each $m < m' < n$, $A_m < A_{m'}$;
    \item for each $b \in \prod_{m < n} A_m$, we have $q_b \in G$, and hence
    $c(b) = k$.
  \end{itemize}
  Therefore, $\langle A_m \mid m < n \rangle$ witnesses this instance of the theorem.
\end{proof}

\begin{remark}
  With some appropriate bookkeeping, the order type $\omega + 1$ in the statement
  of Theorem \ref{polarized_forcing_theorem} can be replaced by any countable
  ordinal $\alpha$.
\end{remark}

\begin{corollary}
  The statement "$\forall n \in [1, \omega) ~ (\Theta_n = \aleph_n)$" is compatible with an arbitrarily 
  large value of the continuum. In particular, if $V$ is a model of $\mathrm{GCH}$, $\chi$ is an 
  infinite cardinal, and $\bb{P}$ is the forcing to add $\chi$-many Cohen reals, then, in $V^{\bb{P}}$, 
  $\Theta_n = \aleph_n$ for all $1 \leq n < \omega$.
\end{corollary}

\begin{proof}
  In $V$, since $\mathrm{GCH}$ holds, we have $\beth_{n-1}^+ = \aleph_n$ for all $1 \leq n < \omega$. 
  Therefore, Theorem \ref{polarized_forcing_theorem} implies that $\Theta_n \leq \aleph_n$ in 
  $V^{\bb{P}}$ for all $1 \leq n < \omega$. By Corollary \ref{theta_cor_1}, it follows that 
  $\Theta_n = \aleph_n$ for all $1 \leq n < \omega$ in $V^{\bb{P}}$.
\end{proof}

\section{A variation, and monochromatic sumsets of reals} \label{sumset_section}

In this section, we discuss an alternative form of higher-dimensional $\Delta$-system
that has appeared in the literature. The following theorem is due to Shelah and
follows from the proof of \cite[Lemma 4.1]{shelah_sierpinski_ii} (cf.\ also
\cite[Claim 7.2.a]{dzamonja_larson_mitchell} and \cite[Lemma 3.6]{zhang_halpern_lauchli}
for more complete proofs of similar statements).

\begin{theorem} \label{variation_lemma}
  Suppose that $\nu \leq \lambda \leq \mu$ are infinite cardinals, $1 \leq n < \omega$, and
  $\mu \rightarrow (\lambda)^{2n}_{2^\nu}$. Suppose moreover that
  $\langle u_a \mid a \in [\mu]^n \rangle$ is a sequence of elements from
  $[\mathrm{On}]^{\leq \nu}$. Then there is $H \in [\mu]^\lambda$ and a
  sequence $\langle u^*_a \mid a \in [H]^{\leq n} \rangle$ of elements from
  $[\mathrm{On}]^{\leq \nu}$ such that
  \begin{enumerate}
    \item $u^*_a \supseteq u_a$ for all $a \in [H]^n$;
    \item for all $a,b \in [H]^n$, we have $\tp(u^*_a, u_a) = \tp(u^*_b, u_b)$;
    \item for all $a,b \in [H]^{\leq n}$, we have $u^*_a \cap u^*_b = u^*_{a \cap b}$;
    \item for all $a_0 \subseteq a_1$ and $b_0 \subseteq b_1$, where
    $a_1, b_1 \in [H]^{\leq n}$, if $\tp(a_1, a_0) = \tp(b_1, b_0)$, then
    $\tp(u^*_{a_1}, u^*_{a_0}) = \tp(u^*_{b_1}, u^*_{b_0})$.
  \end{enumerate}
\end{theorem}

It is currently unclear whether arguments similar to those in the proof of
Theorem \ref{general_theorem} can be used to obtain the conclusion of
Theorem \ref{variation_lemma} from a weaker assumption on $\mu$, such as
$\mu \geq \sigma(\lambda, n)$. It is the case, however, that certain results
that have been proven using Theorem \ref{variation_lemma} can be proven by
instead using Theorem \ref{general_theorem}. This can yield some improvements,
since Theorem \ref{general_theorem} places weaker assumptions on
the cardinal $\mu$. We give one example of such a result here.

In \cite{zhang_monochromatic}, Zhang uses Theorem \ref{variation_lemma} to
prove that, in the forcing extension obtained by adding $\beth_\omega$-many
Cohen reals, we have $\bb{R} \rightarrow^+ (\aleph_0)_r$ for every $r < \omega$, i.e.,
for every $r < \omega$ and every function $f:\bb{R} \rightarrow r$, there is an
infinite set $X \subseteq \bb{R}$ such that $f \restriction (X + X)$ is constant.
We remark that, by a result of Hindman, Leader, and Strauss \cite{hindman_leader_strauss},
if $2^{\aleph_0} < \aleph_\omega$, then there is $r < \omega$ such that
$\bb{R} \not\rightarrow^+ (\aleph_0)_r$, so, over a model of $\mathrm{GCH}$,
it is necessary to add at least $\beth_\omega$-many reals to obtain
$\bb{R} \rightarrow^+ (\aleph_0)_r$ for every $r < \omega$.

Let us examine, though, the number of reals that must be added to obtain
$\bb{R} \rightarrow^+ (\aleph_0)_r$ for some fixed $r < \omega$.
Zhang in fact proves that $\bb{R} \rightarrow^+ (\aleph_0)_2$ holds in $\mathrm{ZFC}$
and, for a fixed $r > 2$, in proving that $\bb{R} \rightarrow^+ (\aleph_0)_r$
holds in the forcing extension, Theorem \ref{variation_lemma} is employed
with $\nu = \aleph_0$, $\lambda = \aleph_1$, and $n = 2r$. Hence, $\mu$ can
be taken to be least such that $\mu \rightarrow (\aleph_1)_{2^{\aleph_0}}^{4r}$.
By the Erd\H{o}s-Rado theorem, then, we can take $\mu = \beth_{4r}^+$.
Zhang's proof uses the fact that we have added at least $\mu$-many Cohen reals
and therefore shows that, for this fixed value of $r > 2$, the statement
$\bb{R} \rightarrow^+ (\aleph_0)_r$ holds in the forcing extension obtained by
adding $\beth_{4r}^+$-many Cohen reals.

Inspection of Zhang's proof reveals that Theorem \ref{general_theorem}, with
$\kappa = \aleph_1$, $\lambda = \beth_1^+$, and $n = 2r$, can be used in place of
Theorem \ref{variation_lemma}. We can therefore take $\mu = \sigma(\beth_1^+, 2r) =
\beth_{2r}^+$, obtaining the following corollary:

\begin{corollary} \label{additive_cor}
  Suppose that $2 < r < \omega$ and $\bb{P}$ is the forcing to add at least
  $\beth_{2r}^+$-many Cohen reals. Then, in $V^{\bb{P}}$, we have
  $\bb{R} \rightarrow^+ (\aleph_0)_r$.
\end{corollary}

This is an improvement on the bound of $\beth_{4r}^+$ given by Zhang's proof,
though of course it does not improve on Zhang's bound for obtaining
$\bb{R} \rightarrow^+ (\aleph_0)_r$ simultaneously for all $r < \omega$.
We omit the adaptation of Zhang's proof using Theorem
\ref{general_theorem} instead of Theorem \ref{variation_lemma} here, as it would
entail introducing a considerable number of definitions and only involves
very minor changes to Zhang's proof. Instead, we direct the reader to
\cite{zhang_monochromatic} and \cite{sumset_proof_alteration}, in which Zhang's
original proof and the adaptation using Theorem \ref{general_theorem} are
spelled out in detail.

\medskip

\noindent \small{\textbf{Data Availability:} Data sharing not applicable to this article as no datasets
were generated or analysed during the current study.}

\bibliographystyle{plain}
\bibliography{bib}

\end{document}